\documentclass{amsart}

\title[Rational forms via the $R$-matrix formalism]{
The $R$-matrix presentation for the rational form of a quantized enveloping algebra
}

\author[M. Rupert]{Matthew Rupert}
\address{Department of Mathematics and Statistics, University of Saskatchewan}
\email{matthew.rupert@usask.ca}

\author[C. Wendlandt]{Curtis Wendlandt}
\address{Department of Mathematics and Statistics, University of Saskatchewan}
\email{wendlandt@math.usask.ca}


\usepackage{amsmath,amssymb} 
\usepackage{mathrsfs} 
\usepackage[shortalphabetic]{amsrefs} 
\usepackage[usenames,dvipsnames]{xcolor} 
\usepackage{bbm} 
\usepackage[scr=boondoxo]{mathalfa} 
\usepackage{dsfont} 
\usepackage[
	colorlinks,
	plainpages,
	citecolor=red!50!black,
    filecolor=Darkgreen,
    linkcolor=blue!40!black,
    urlcolor=cyan!50!black!90]{hyperref} 
\usepackage{enumitem}
\usepackage{version} 
\usepackage{tikz-cd} 
\usepackage{accents} 
\usepackage{stmaryrd} 
\usepackage{graphicx}
\usepackage{upgreek}

\newtheorem{theorem}{Theorem}[section]
\newtheorem{proposition}[theorem]{Proposition}
\newtheorem{corollary}[theorem]{Corollary}
\newtheorem{lemma}[theorem]{Lemma}

\newtheorem{theoremintro}{Theorem}

\theoremstyle{definition}
\newtheorem{definition}[theorem]{Definition}
\newtheorem{remark}[theorem]{Remark}

\newcommand{\End}{\mathrm{End}}

\newcommand{\Hom}{\mathrm{Hom}}

\newcommand{\iso}{\xrightarrow{\,\smash{\raisebox{-0.5ex}{\ensuremath{\scriptstyle\sim}}}\,}}

\newcommand{\into}{\hookrightarrow}
\newcommand{\onto}{\twoheadrightarrow}

\newcommand{\mfb}{\mathfrak{b}}

\newcommand{\mfg}{\mathfrak{g}}
\newcommand{\mfh}{\mathfrak{h}}

\newcommand{\mfn}{\mathfrak{n}}

\newcommand{\mfgl}{\mathfrak{g}\mathfrak{l}}
\newcommand{\mfsl}{\mathfrak{s}\mathfrak{l}}

\newcommand{\mcJ}{\mathcal{J}}

\newcommand{\mcV}{\mathcal{V}}

\newcommand{\mbI}{\mathbf{I}}

\newcommand{\mbU}{\mathbf{U}}

\newcommand{\mbbA}{\mathbb{A}}

\newcommand{\mbbL}{\mathbb{L}}

\newcommand{\mbbV}{\mathbb{V}}

\newcommand{\C}{\mathbb{C}}
\newcommand{\Q}{\mathbb{Q}}

\newcommand{\Z}{\mathbb{Z}}
\newcommand{\N}{\mathbb{N}}

\newcommand{\mrL}{\mathrm{L}}
\newcommand{\mrR}{\mathrm{R}}
\newcommand{\mrT}{\mathrm{T}}

\newcommand{\msL}{\mathsf{L}}

\newcommand{\msQ}{\mathsf{Q}}

\newcommand{\msS}{\mathsf{S}}

\newcommand{\msT}{\mathsf{T}}

\newcommand{\eps}{\epsilon}
\newcommand{\veps}{\varepsilon}

\allowdisplaybreaks
\numberwithin{equation}{section}

\newcommand{\URg}{\mathbf{U}_{\mathrm{R}}^{\lambda}(\mathfrak{g})}
\newcommand{\bUqg}{\mathbf{U}_q^\lambda(\mathfrak{g})}
\newcommand{\Lz}{\mathscr{L}}
\newcommand{\sbinom}[2]{\begin{bmatrix}{#1}\\{#2}\end{bmatrix}}

\newcommand{\wlattice}{\mathsf{\Lambda}}
\newcommand{\olambda}{|\![\lambda]\!|}

\newcommand{\Xr}{\mathrm{X}}
\newcommand{\Yr}{\mathrm{Y}}
\newcommand{\chev}{\theta}
\newcommand{\Rp}{\mathrm{R}^+}
\newcommand{\Rd}{\mathscr{D}}
\newcommand{\GLgr}{\mathrm{GL}^\mfh(V)}

\DeclareMathAlphabet\EuScript{U}{eus}{m}{n}
\SetMathAlphabet\EuScript{bold}{U}{eus}{b}{n}
\newcommand{\scL}{\EuScript{L}}
\newcommand{\scK}{\EuScript{K}}
\newcommand{\scE}{\EuScript{E}}

\begin{document}

\begin{abstract}
Let $U_q(\mathfrak{g})$ denote the rational form of the quantized enveloping algebra associated to a complex simple Lie algebra $\mathfrak{g}$. Let $\lambda$ be a nonzero dominant integral weight of $\mathfrak{g}$, and let $V$ be the corresponding type $1$ finite-dimensional irreducible representation of $U_q(\mfg)$. Starting from this data, the $R$-matrix formalism for quantum groups outputs a Hopf algebra $\URg$ defined in terms of a pair of generating matrices satisfying well-known quadratic matrix relations. In this paper, we prove that this Hopf algebra admits a  Chevalley--Serre type presentation which can be recovered from that of $U_q(\mfg)$ by adding a single invertible quantum Cartan element. We simultaneously establish that $\URg$ can be realized as a Hopf subalgebra of the tensor product of the space of Laurent polynomials in a single variable with the quantized enveloping algebra associated to the lattice generated by the weights of $V$.  The proofs of these results are based on a detailed analysis of the homogeneous components of the matrix equations and generating matrices defining $\URg$, with respect to a natural grading by the root lattice of $\mfg$ compatible with the weight space decomposition of $\End(V)$. 
\end{abstract}

\maketitle

{\setlength{\parskip}{0pt}
\setcounter{tocdepth}{1} 
\tableofcontents
}

\section{Introduction}\label{sec:Intro}


 This article is a sequel to the authors work \cite{GRW-RQUE} with S. Gautam, which addressed the problem of rebuilding the quantized universal enveloping algebra $U_\hbar(\mfg)$ of a semisimple complex Lie algebra $\mfg$, viewed as a topological Hopf algebra over $\C[\![\hbar]\!]$, from a solution $\mrR$ of the quantum Yang--Baxter equation associated to any of its non-trivial finite-dimensional representations $\mcV$. In more detail, the solution to this problem presented in \cite{GRW-RQUE} passes through a version of the  Faddeev--Reshetikhin--Takhtajan formalism for constructing quantum groups \cite{FRT}, which produces a Hopf algebra $\mathrm{U}_\mrR(\mfg)$ whose generators are encoded by a pair of matrices $\mrT^+$ and $\mrT^-$ which, in particular, have the property that $\mrL^\pm:=I+\hbar\mrT^\pm$ satisfy the celebrated matrix equations 
\begin{equation}\label{RLL:I}
\mrR\mrL_1^\pm \mrL_2^\pm =\mrL_2^\pm \mrL_1^\pm \mrR \quad \text{ and }\quad 
\mrR \mrL_1^+ \mrL_2^-=\mrL_2^-\mrL_1^+\mrR.
\end{equation}
In Theorem 5.7 of \cite{GRW-RQUE}, it was shown that $\mathrm{U}_\mrR(\mfg)$ is isomorphic to the tensor product of the Drinfeld double of the Borel subalgebra $U_\hbar(\mfb^+)\subset U_\hbar(\mfg)$ with a commutative Hopf algebra whose definition is encoded by the space of $\mfg$-invariants of the classical limit of $\mcV$. This was then used to deduce that $U_\hbar(\mfg)$ can be recovered both as a quotient of $\mathrm{U}_\mrR(\mfg)$ by the ideal generated by the coefficients of certain central matrices, and as the fixed-point subalgebra of $\mathrm{U}_\mrR(\mfg)$ with respect to a natural family of automorphisms; see \cite{GRW-RQUE}*{Thm.~5.14}. 

Crucially, the proofs of these results make extensive use of deformation arguments which reduce many of the key statements to simpler results for the Lie algebra $\mfg$ itself, established in \cite{GRW-RQUE}*{\S3}. We note, however, that these arguments no longer directly apply when $U_\hbar(\mfg)$ is replaced by its rational form $U_q(\mfg)$ defined over $\Q(q)$ (see Definition \ref{D:Uqg}) and, in addition, they do not explain how to derive the defining Chevalley--Serre relations of the quantized enveloping algebra of $\mfg$ from the quadratic matrix equations at the heart of the $R$-matrix formalism. 
The main goal of this article is to remedy this issue, with focus on the case where the underlying representation is irreducible.

\subsection{Main results}\label{ssec:results}

We now summarize our main results in detail. Let  $\lambda$ be a fixed nonzero dominant integral weight of $\mfg$ and let $V=V(\lambda)$ be the associated type $1$ finite-dimensional irreducible representation of $U_q(\mfg)$; see Section \ref{ssec:Uqg-reps}.  Using that $U_q(\mfg)$ is nearly quasitriangular, one obtains a  distinguished solution $\mrR\in \End(V\otimes V)$ of the quantum Yang--Baxter equation; see \eqref{R:def}.  Starting from this $R$-matrix, we define a Hopf algebra $\URg$ over $\Q(q)$, graded by the root lattice $\msQ$ of $\mfg$, whose generators are encoded by matrices
\begin{equation*}
\mrL^\pm\in \bigoplus_{\beta\in \msQ_\pm}\End(V)_{\pm\beta}\otimes \URg,
\end{equation*}
where $\msQ_+$ and $\msQ_-$ are the positive and negative cones in $\msQ$, respectively, and $\End(V)_\beta$ is the $\beta$-weight space of the $U_q(\mfg)$-module $\End(V)$. These generating matrices are subject to the quadratic matrix equations \eqref{RLL:I}
in addition to the requirement that their weight zero components $\Lz^+$ and $\Lz^-$ are mutual inverses. 

In this article, we make the relation between $\URg$ and $U_q(\mfg)$ precise without relying on the deformation arguments of \cite{GRW-RQUE}. To explain this relation, 
let $\wlattice(\lambda)$ be the submodule of the weight lattice of $\mfg$ generated by the weights of $V$. As $\lambda$ is nonzero, this is a genuine lattice containing $\msQ$, and we may thus associate to it a quantized enveloping algebra $U_q^\lambda(\mfg)\supset U_q(\mfg)$; see Definition \ref{D:Uqg} and Lemma \ref{L:lattice-V}. The following theorem, which is a combination of  Proposition \ref{P:gen-rels} and Theorem \ref{T:main}, is the main result of this article.
\begin{theoremintro}\label{T:intro}
There is an embedding of $\msQ$-graded Hopf algebras
\begin{equation*}
\Upupsilon:\URg \into U_q^{\lambda}(\mfg)\otimes \Q(q)[v^{\pm 1}]
\end{equation*}
with image $\bUqg$ which can be identified with the unital associative $\Q(q)$-algebra generated by $\{\upxi_i^\pm,x_i^\pm\}_{i\in \mbI}\cup\{\upxi_\lambda^\pm\}$, subject to the relations
\begin{gather*}
\upxi_i^\pm\upxi_i^\mp=1=\upxi_\lambda^\pm\upxi_\lambda^\mp, \quad[\upxi_i^\pm, \upxi_j^\pm]=0=[\upxi_i^\pm,\upxi_\lambda^\pm],
\\
\upxi_i^+ x_j^\pm \upxi_i^-=q^{\pm (\alpha_i,\alpha_j)}x_j^\pm, \quad \upxi_\lambda^+ x_j^\pm \upxi_\lambda^-=q^{\pm (\lambda,\alpha_j)}x_j^\pm,
\\
[x_i^+,x_j^-]=\delta_{ij}\frac{\upxi_i^+-\upxi_i^-}{q_i-q_i^{-1}},
\\
\sum_{b=0}^{1-a_{ij}}(-1)^b \sbinom{1-a_{ij}}{b}_{q_i} (x_i^\pm)^b x_j^\pm (x_i^\pm)^{1-a_{ij}-b}=0 \quad \text{ for }\; i\neq j,
\end{gather*}
where  $(a_{ij})_{i,j\in \mbI}$ is the Cartan matrix  of $\mfg$, $\{\alpha_i\}_{i\in \mbI}$ is a basis of simple roots, $(\,,\,)$ is a fixed invariant form on $\mfg$, and $q_i=q^{d_i}$ with $d_i$ the $i$-th symmetrizing integer; see Section \ref{ssec:g}. 
\end{theoremintro}
The Hopf algebra $\bUqg$ is first defined in Section \ref{ssec:Uqg-basic}, where it is also proven to admit the Chevalley--Serre type presentation given in the statement of the above theorem; see Definition \ref{D:bUqg} and Proposition \ref{P:gen-rels}. It satisfies 
\begin{equation*}
U_q(\mfg)\otimes \Q(q)[v_\lambda^{\pm 1}]\subset \bUqg\subset U_q^{\lambda}(\mfg)\otimes \Q(q)[v^{\pm 1}
],
\end{equation*}
where $v_\lambda=v^{n}$ is the minimal positive power of $v$ contained in $\bUqg$; explicitly, $n$ is equal to the order $|\![\lambda]\!|$ of the class of $\lambda$ in the quotient group $\wlattice(\lambda)/\msQ$. Moreover, as will be explained in Proposition \ref{P:bUqg-str}, one has 
\begin{equation*}
\bUqg/(v_\lambda-1)\cong U_q^\lambda(\mfg).
\end{equation*}

 When $\lambda$ is the first fundamental weight of $\mfg=\mfsl_n$, $V$ can be identified with the vector representation $\Q(q)^n$ and $\bUqg$ coincides with the quantized enveloping algebra $U_q(\mfgl_n)$ of the general linear algebra $\mfgl_n$ \cites{DF93,Jimbo86}. In this case, the above theorem reduces to the statement of Theorem 2.1 in \cite{DF93}. More generally, when $\mfg$ is of classical type and $V$ is the vector representation of $U_q(\mfg)$, Theorem \ref{T:intro} reduces to the finite type counterpart of \cite{MiHa-L98}*{Thm.~1}; see also \cites{JLM20-c,JLM20-b}. 
In these special cases, it has been known since the foundational work \cite{FRT} of Faddeev, Reshetikhin and Takhtajan that there exists an invertible central element $\EuScript{Z}\in \URg$ such that
 \begin{equation*}
 \URg/(\EuScript{Z}-1)\cong U_q^\lambda(\mfg).
 \end{equation*}
When $\mfg$ is the special linear algebra $\mfsl_n$, the element $\EuScript{Z}$ can be taken to be the $n$-fold product of the  commuting elements $\mathscr{l}_\mu$ associated to the $n$ weights of $\Q(q)^n$; see Remark 21 of \cite{FRT}, in addition to Theorem 8.33 of \cite{KS-book} and Section 7.1 of \cite{GRW-RQUE}. If instead $\mfg$ is of symplectic or orthogonal type, then it follows from Propositions 3.3 of \cite{JLM20-c}  and \cite{JLM20-b} that there is an invertible central element $\EuScript{Z}\in \URg$ and a diagonal matrix $D$, with entries that are integer powers of $q$, satisfying 
\begin{equation*}
D(\mrL^\pm)^t  D^{-1} \mrL^\pm =\mrL^\pm  D (\mrL^\pm)^tD^{-1}= \EuScript{Z}^{\pm 1}\, I,
\end{equation*}
where $I$ is the identity matrix. By Theorem  8.33 of \cite{KS-book}, imposing the additional relation $\EuScript{Z}=1$ recovers $U_q^\lambda(\mfg)$; see also Theorem 12 and Remark 21 of \cite{FRT}. An equivalent characterization of this element $\EuScript{Z}$ is encoded by the constant term of the left-hand side of \cite{MiHa-L98}*{(3.7)}.

This story was extended to the case where $V$ is the seven dimensional irreducible representation of $U_q(\mathsf{G}_2)$ in Theorem 11 of \cite{Sasaki}.  The generalization of these results to any Cartan type and dominant integral weight $\lambda$ is stated explicitly in Corollary \ref{C:Recover}, where a uniform expression for $\EuScript{Z}$ is also given.

\subsection{The inverse of $\Upupsilon$}\label{ssec:results-II} 

The construction of the homomorphism $\Upupsilon$ from Theorem \ref{T:intro} follows a well-known scheme which has been applied in \cites{DF93,MiHa-L98,WRTT,GRW-RQUE}, for instance, and is inspired by \cite{FRT}*{Thm.~18} and \cite{Dr}*{Thm.~6}. Namely,  $\mrL^\pm$ are sent to certain generating matrices $\EuScript{L}^\pm\in \End(V)\otimes \bUqg$ for $\bUqg$ which play the role of \textit{L-operators}. They are constructed using $V$ and the quasi universal $R$-matrix $\mathscr{R}^+$ of $U_q(\mfg)$, and satisfy the defining relations of $\URg$ due to properties of $\mathscr{R}^+$; see Proposition \ref{P:bUqg-L}.

To prove that $\Upupsilon$ is an isomorphism, we deviate from the approach of \cite{GRW-RQUE} and  construct $\Upupsilon^{-1}$ explicitly by carrying out a uniform, type independent, analysis of the components in the Gauss decompositions of $\mrL^+$ and $\mrL^-$. The two  results at the heart of this analysis are Theorems \ref{T:Lz} and \ref{T:Gen}. In Theorem \ref{T:Lz}, it is shown that the weight zero component $\Lz^\pm$ of $\mrL^\pm$ decomposes in the diagonal form 
\begin{equation*}
\Lz^\pm = \sum_{\mu} \mathrm{Id}_{V_\mu}\otimes \mathscr{l}_\mu^{\pm 1},
\end{equation*}
where the summation is taken over the finite set of weights of $V$, and the elements $\mathscr{l}_\mu$ pairwise commute.  
We note that in all the special cases alluded to at the end of Section \ref{ssec:results}, the weight spaces $V_\mu$ of $V$ are one-dimensional and the above decomposition for $\Lz^\pm$ follows automatically from the fact that it is of weight zero. Theorem \ref{T:Lz} provides a non-trivial generalization of this observation which, as a consequence of the definition of $\EuScript{L}^\pm$ given in Section \ref{ssec:bUqg-gen}, must hold if the identification of $\mrL^\pm$ with  $\EuScript{L}^\pm$ is to yield an isomorphism between $\URg$ and  $\bUqg$.

The proof of Theorem \ref{T:Lz} implies that if $\mu$ and $\gamma$ are any two weights of $V$ then the ratio $\mathscr{l}_\gamma^{-1}\mathscr{l}_\mu$ only depends on the difference $\gamma-\mu\in \msQ$. Since $\lambda$ is nonzero, each simple root $\alpha_i$ arises as such a difference and, consequently, there is a unique element $\upxi_i\in \URg$, for each $i\in \mbI$, such that
\begin{equation*}
\upxi_i=\mathscr{l}_{\mu+\alpha_i}^{-1}\mathscr{l}_{\mu}
\end{equation*}
for any weight $\mu$ of $V$ for which $\mu+\alpha_i$ is also a weight; see Corollary \ref{C:upxi}. 
We then prove in Theorem \ref{T:Gen} that the components $\mbbL^+_{\alpha_i}$ and $\mbbL_{-\alpha_i}^-$ of the unipotent matrices $\mbbL^+=\Lz^-\mrL^+$ and $\mbbL^-=\mrL^- \Lz^+$ associated to a simple root $\alpha_i$ are always  pure tensors of the form 
\begin{equation*}
\mbbL_{\alpha_i}^+ =\pi(E_i)\otimes \Xr_i \quad \text{ and }\quad \mbbL_{-\alpha_i}^- =\pi(F_i)\otimes \Yr_i,
\end{equation*}
where
$E_i$ and $F_i$ are the standard Chevalley generators of $U_q(\mfg)$ (see Definition \ref{D:Uqg}) and $\pi:U_q(\mfg)\to \End(V)$ is the action homomorphism. In addition, it is shown that all coefficients of $\mbbL^+$ and  $\mbbL^-$ are contained in the subalgebras generated by $\{\Xr_i\}_{i\in \mbI}$ and $\{\Yr_i\}_{i\in \mbI}$, respectively.  

Using these two theorems, we show in the proof of Theorem \ref{T:main} that the assignment
\begin{gather*}
\upxi_i^{\pm}\mapsto \upxi_i^{\pm 1}, \quad \upxi_\lambda^{\pm}\mapsto\mathscr{l}_\lambda^{\mp 1},\quad x_i^+\mapsto \frac{\Xr_i}{q_i^{-1}-q_i} \quad \text{ and }\quad x_i^-\mapsto \frac{\Yr_i}{q_i-q_i^{-1}}
\quad \forall \; i\in \mbI
\end{gather*}
uniquely extends to an algebra homomorphism $\bUqg\to \URg$ which is in fact equal to $\Upupsilon^{-1}$. The main additional difficulty in establishing this assertion lies in showing that the above assignment preserves the $q$-Serre relations of $\bUqg$.  This is established in Corollary \ref{C:Serre} using versions of the defining matrix equations of $\URg$ associated to arbitrary tensor powers of $V$. 

\subsection{Remarks}\label{ssec:remarks}

Let us now give a few concluding remarks. 
Firstly, the results of this article and their proofs are readily seen to remain valid when $\Q(q)$ is replaced by an arbitrary field $\Bbbk$ of characteristic zero, provided $q\in \Bbbk$ is taken to be a nonzero element which is not a root of unity. Part of our motivation for specifying the choice of base field $\Q(q)$ lies in the fact that $\URg$ should admit a natural theory of integral forms over $\mathbb{A}=\Q[q,q^{-1}]$, consistent with the narrative for the quantum enveloping algebra of $\mfgl_n$ from \cite{FiTs19}*{\S3.1}. Though we do not endeavour to develop this theory fully in the present article, we do wish to make a few preliminary observations. 

Let $\mathbb{V}$ denote the free $\mbbA$-submodule of $V$ spanned by its canonical basis. Then, by \cite{Lusztig-Book}*{Cor.~24.1.5}, the $R$-matrix $\mrR\in \End(V\otimes V)$ preserves $\mathbb{V}\otimes_{\mathbb{A}}\!\mathbb{V}$. We may thus define a unital associative $\mbbA$-algebra ${}_\mathbb{A}\!\URg$ by replacing $\End(V)$ by the $\msQ$-graded free $\mbbA$-module $\End_\mbbA(\mbbV)$ in the definition of $\URg$ (see Definition \ref{D:URg}). This is a Hopf algebra over $\mbbA$ which specializes at $q=1$ to a commutative, but non-cocommutative, Hopf algebra over $\Q$ and satisfies
\begin{equation*}
{}_\mathbb{A}\!\URg\otimes_{\mbbA} \Q(q)\cong \URg.
\end{equation*}
 In addition, it admits a (conjecturally injective) homomorphism $\imath: {}_\mathbb{A}\!\URg \to \URg$ with image equal to the $\mbbA$-subalgebra of $\URg$ generated by the matrix coefficients of $\mrL^\pm$, viewed as elements of $\End_\mbbA(\mbbV)\otimes_{\mbbA}\URg$. Furthermore, it follows from the explicit formulas for the universal quasi $R$-matrix of $U_q(\mfg)$ obtained in \cites{KR90, LeSo90} (see also \cite{CPBook}*{\S8.3}) that the composite $\Upupsilon \circ \imath$ satisfies 
 \begin{equation*}
 (\Upupsilon \circ \imath)({}_\mathbb{A}\!\URg)\subset \EuScript{A}_q(\mfg)\otimes_\mbbA \mbbA[v^{\pm 1}],
 \end{equation*}
 where $\EuScript{A}_q(\mfg)\subset U_q^\lambda(\mfg)$ is the quantum group introduced by De Concini and Procesi in \cite{DCP-survey}*{\S12}. In more detail, it is the minimal $\mbbA$-subalgebra of $U_q^\lambda(\mfg)$  which is stable under Lusztig's braid group operators and contains the elements $(q_i-q_i^{-1})E_i$ and $(q_i-q_i^{-1})F_i$ for each $i\in \mbI$; see in particular Theorem 12.1 and (11.8.1) of \cite{DCP-survey}.

Next, let us elaborate on the relationship between the results of this article and those of \cite{GRW-RQUE}, which were briefly summarized at the beginning of Section \ref{sec:Intro}.
Let $\mcV$ be a non-trivial indecomposable representation of the quantized enveloping algebra $U_\hbar(\mfg)$ which is both free and of finite rank as a  $\C[\![\hbar]\!]$-module.
 Let  $\mathrm{U}_\mrR(\mfg)$ be the $R$-matrix algebra associated to $\mcV$, as defined in Section 5.1 of \cite{GRW-RQUE}, with generating matrices $\mrT^+$ and $\mrT^-$. Let $\mathbb{J}$ be the two-sided ideal of $\mathrm{U}_\mrR(\mfg)$  topologically generated by the coefficients of the matrices 
\begin{equation*}
\mathscr{T}^++\mathscr{T}^-+\hbar \mathscr{T}^+\mathscr{T}^- \quad \text{ and }\quad \mathscr{T}^++\mathscr{T}^-+\hbar \mathscr{T}^-\mathscr{T}^+,
\end{equation*}
where $\mathscr{T}^\pm$ is the weight zero component of $\mrT^\pm$. We note that the ideal $\mathbb{J}$ appears indirectly in Remark 6.6 of \cite{GRW-RQUE}, and is engineered so that the matrices $\Lz^\pm:=I+\hbar \mathscr{T}^\pm$ satisfy $\Lz^+\Lz^-=I=\Lz^-\Lz^+$ in $\mathrm{U}_\mrR(\mfg)/\mathbb{J}$.
Given these preliminaries, the counterpart of Theorem \ref{T:intro} in the setting of \cite{GRW-RQUE} is the assertion that there is an isomorphism of topological Hopf algebras 
\begin{equation}\label{h-adic-decomp}
\mathrm{U}_\mrR(\mfg)/\mathbb{J}\cong U_\hbar(\mfg)\, \widehat{\otimes} \,\C[z][\![\hbar]\!],
\end{equation}
where $\widehat{\otimes}$ is the topological tensor product of $\C[\![\hbar]\!]$-modules and the Hopf structure on the right-hand side is uniquely determined by the conditions that $U_\hbar(\mfg)$ is a Hopf subalgebra and $z$ is a primitive element. This isomorphism can be viewed as an intermediate result between Theorems 5.7 and 5.14 of \cite{GRW-RQUE}, and is readily deduced from them. 
One significant difference between it and the identification of $\URg$ with $\bUqg$ provided by Theorem \ref{T:intro} is that the right-hand side of \eqref{h-adic-decomp} does not depend on the lattice generated by the set of weights of $\mcV$. In contrast,  $\bUqg$ is controlled by the class of $\lambda$ in the quotient group $\wlattice(\lambda)/\msQ$, and coincides with the tensor product $U_q^\lambda(\mfg)\otimes\Q(q)[v^{\pm 1}]$ if and only if $\lambda$ lies in the root lattice $\msQ$; see Proposition \ref{P:bUqg-str}. 
 Despite these observations,  the statements of Theorems 5.7 and 5.14 of \cite{GRW-RQUE} for indecomposable $\mcV$ can themselves be proven using the analysis outlined in Section \ref{ssec:results-II}, which more directly exploits the  equations at the heart of the $R$-matrix formalism.

Finally, whereas the main constructions of \cite{GRW-RQUE} took as input an arbitrary non-trivial finite-dimensional representation of $U_\hbar(\mfg)$, we have narrowed our focus to those non-trivial finite-dimensional representations of $U_q(\mfg)$ which are both irreducible and of type $1$. This has allowed us to develop a uniform theory labelled by nonzero dominant integral weights which recovers all of the concrete examples that have been considered in the literature as special cases. That being said, our expectation is that Theorem \ref{T:intro} admits a generalization which takes as input any type $1$ finite-dimensional representation of $U_q(\mfg)$ with a non-trivial irreducible summand, shares many common features with its counterparts from \cite{GRW-RQUE}, and can be proven using natural generalizations of Theorems \ref{T:Lz} and \ref{T:Gen}.

\subsection{Outline}\label{ssec:I-Outline}

In Section \ref{sec:Prelim}, we review a number of basic facts from the representation theory of $U_q(\mfg)$, including the definitions and main properties of the $R$-matrices and representation spaces which feature throughout the paper. Section \ref{sec:Uqlambda} is devoted to introducing the quantum group $\bUqg$ and studying its algebraic structure. Sections \ref{sec:URg} and \ref{sec:URg=Uq} contain the main results of this article. The former is focused entirely on the $R$-matrix algebra $\URg$, which is first defined in Definition \ref{D:URg}. The main results of this section (Theorems \ref{T:Lz} and \ref{T:Gen}) show that the coefficients of the weight zero and unipotent parts of the generating matrices $\mrL^\pm$ are remarkably restricted, and are used to identify Chevalley--Serre type generators in $\URg$. In Section \ref{sec:URg=Uq} we prove Theorem \ref{T:main}, which establishes that $\URg$ and $\bUqg$ are one and the same. In addition, we use this theorem to explain in Section \ref{ssec:Cons-R} how to characterize $U_q(\mfg)$ and  $U_q(\mfg)\otimes \Q(q)[v_\lambda^{\pm 1}]$ as Hopf subalgebras of $\URg$; see Corollaries \ref{C:Recover} and \ref{C:nlambda}.
\subsection{Acknowledgments}

The authors are grateful to Alex Weekes for many insightful discussions and helpful comments throughout the writing of this article. The first author was supported by the Pacific Institute for the Mathematical Sciences (PIMS) Postdoctoral Fellowship Program. The second author gratefully acknowledges the support of the Natural Sciences and Engineering Research Council of Canada (NSERC), provided via the Discovery Grants Program (Grant RGPIN-2022-03298 and DGECR-2022-00440).
\section{Preliminaries on \texorpdfstring{$U_q(\mfg)$}{U\_q(g)}}\label{sec:Prelim}

\subsection{The Lie algebra  \texorpdfstring{$\mfg$}{g}}
\label{ssec:g}

Let $\mfg$ be a simple Lie algebra over the complex numbers equipped with an invariant, non-degenerate, symmetric bilinear form $(\,,\,)$.  Let $\mfh\subset \mfg$ be a Cartan subalgebra and $\{\alpha_i\}_{i\in \mbI}\subset \mfh^\ast$ a basis of simple roots relative to $\mfh$. We normalize $(\,,\,)$ so that the square length of every short simple root is $2$. Let $(a_{ij})_{i,j\in \mbI}$ denote the Cartan matrix of $\mfg$, defined by 
\begin{equation*}
d_ia_{ij}=(\alpha_i,\alpha_j) \quad \forall \quad i,j\in \mbI,
\end{equation*}
where $d_i=\frac{(\alpha_i,\alpha_j)}{2}$ is the $i$-th symmetrizing integer. Let $\{\varpi_i\}_{i\in \mbI}\subset \mfh^\ast$ denote the fundamental weights of $\mfg$, uniquely determined by $(\varpi_i,\alpha_j)=d_j \delta_{ij}$ for all $i,j\in \mbI$, and let  $\msQ=\bigoplus_{i\in \mbI}\Z \alpha_i$ and $\wlattice=\bigoplus_{i\in \mbI}\Z \varpi_i$ denote the root and weight lattices of $\mfg$, respectively. In addition, we will write $\msQ_+$ and $\msQ_-=-\msQ_+$ for the positive and negative cones in the root lattice and $\wlattice_+$ for the set of dominant integral weights of $\mfg$. That is, one has
\begin{equation*}
\msQ_+:=\bigoplus_{i\in \mbI}\Z_{\geq 0}\alpha_i \quad \text{ and }\quad \wlattice_+:=\bigoplus_{i\in \mbI}\Z_{\geq 0} \varpi_i.
\end{equation*}
Throughout the entire course of this paper, $\lambda\in \wlattice_+$ will denote a fixed \textit{nonzero} dominant integral weight of $\mfg$. We shall write $\olambda$ for the order of the equivalence class $[\lambda]$ of $\lambda$ in the quotient group $\wlattice/\msQ$. 

Additionally, given $\beta\in \msQ_+$ we will write $\msQ_+^\beta$ for the finite set consisting of all $\nu\in \msQ_+$ satisfying $\nu\leq \beta$ with respect to the standard partial ordering:
\begin{equation*}
\msQ_+^\beta=\{\alpha\in \msQ_+: \beta-\alpha\in \msQ_+\}.
\end{equation*}
Similarly, for each $\beta\in \dot\msQ_+:=\msQ_+\setminus\{0\}$ we define $\dot\msQ_+^\beta$ to be the finite subset of $\dot\msQ$ consisting of all $\alpha\in \dot\msQ_+$ for which $\beta-\alpha\in \dot\msQ_+$.

\subsection{The Hopf algebra \texorpdfstring{$U_q^{\scE}(\mfg)$}{U\_q(g)}}
\label{ssec:Uqg-basic}

In this section we recall the definition and basic properties of the Drinfeld--Jimbo algebra associated to $\mfg$, freely drawing from the standard references as necessary. We refer the reader to \cites{Lusztig-Book,Jantzen-Book,KS-book} or \cite{CPBook} for further details.

In what follows, we fix $\scE$ to be a $\Z$-submodule of the weight lattice $\wlattice$ containing the root lattice $\msQ$. In addition, we employ the standard notation for Gaussian integers and binominal coefficients: If $m,n,r\in \Z$ with $n\geq r\geq 0$, then 
\begin{gather*}
\sbinom{n}{r}_q=\frac{[n]_q!}{[r]_q![n-r]_q!}, \quad [m]_q!=[m]_q[m-1]_q\cdots [1]_q,\\ 
[m]_q=\frac{q^m-q^{-m}}{q-q^{-1}}.
\end{gather*}
\begin{definition}\label{D:Uqg}
 The quantum enveloping algebra $U_q^{\scE}(\mfg)$ is the unital, associative $\Q(q)$-algebra generated by $\{E_i,F_i\}_{i\in \mbI}$ and $\{K_\mu^{\pm 1}\}_{\mu\in \scE}$, subject to the relations: 
\begin{gather*}
K_\mu K_\gamma = K_{\mu+\gamma}, \quad K_0=1, 
\\
K_\mu E_j K_\mu^{-1}=q^{(\mu,\alpha_j)}E_j, \quad K_\mu F_j K_\mu^{-1}=q^{-(\mu,\alpha_j)}F_j,
 \\
[E_i,F_j]=\delta_{ij}\frac{K_{\alpha_i}-K_{-\alpha_i}}{q_i-q_i^{-1}},
\\ 
\sum_{b=0}^{1-a_{ij}}(-1)^b \sbinom{1-a_{ij}}{b}_{q_i} E_i^b E_j E_i^{1-a_{ij}-b}=0,
\\
\sum_{b=0}^{1-a_{ij}}(-1)^b \sbinom{1-a_{ij}}{b}_{q_i} F_i^b F_j F_i^{1-a_{ij}-b}=0,
\end{gather*}
where $q_i=q^{d_i}$ and $i\neq j$ in the last two relations.
\end{definition}
The algebra $U_q^\scE(\mfg)$ is a Hopf algebra over $\Q(q)$ with coproduct $\Delta$, antipode $S$, and counit $\veps$ determined by
\begin{gather*}
\Delta(K_\mu)=K_\mu \otimes K_\mu, \quad S(K_\mu)=K_{-\mu}, \quad \veps(K_\mu)=1,\\
\Delta(E_i)=E_i\otimes K_i + 1\otimes E_i, \quad S(E_i)=-E_i K_i^{-1}, \quad \veps(E_i)=0,\\
\Delta(F_i)=F_i\otimes 1 + K_i^{-1}\otimes F_i, \quad S(F_i)=-K_iF_i,\quad \veps(F_i)=0,
\end{gather*}
for all $\mu \in \scE$ and $i\in \mbI$, where we have set $K_i:=K_{\alpha_i}$ for each $i\in \mbI$, following the usual conventions. Moreover, $U_q^\scE(\mfg)$ admits a $\msQ$-grading compatible with this Hopf structure, with homogeneous components given by
\begin{equation*}
U_q^\scE(\mfg)_\beta=\{x\in U_q^\scE(\mfg): K_\mu x K_\mu^{-1} =q^{(\beta,\mu)}x \quad \forall \; \mu\in \scE\} \quad \forall \quad \beta \in \msQ.
\end{equation*}

Let $U_q(\mfn^+)$ and $U_q(\mfn^-)$ denote the (graded) subalgebras of $U_q^\scE(\mfg)$ generated by $\{E_i\}_{i\in \mbI}$ and $\{F_i\}_{i\in \mbI}$, respectively, and let $U_q^\scE(\mfh)$ denote the subalgebra generated by $\{K_\mu\}_{\mu\in \scE}$. Similarly, the subalgebra of $U_q^\scE(\mfg)$ generated by $U_q^\scE(\mfh)$ and $U_q(\mfn^\pm)$ is denoted $U_q^\scE(\mfb^\pm)$.
Given this notation, the triangular decomposition for $U_q^\scE(\mfg)$ asserts that multiplication induces a $\Q(q)$-linear isomorphism 
\begin{equation}\label{Uqg:TD}
U_q(\mfn^+)\otimes U_q^\scE(\mfh) \otimes U_q(\mfn^-)\iso U_q^\scE(\mfg).
\end{equation}

To conclude this subsection, we recall two variants of the Chevalley involution on $U_q^\scE(\mfg)$ which shall feature throughout this article. Firstly, $U_q^\scE(\mfg)$ admits an algebra involution $\omega$ uniquely determined by 
\begin{equation}\label{omega}
\omega(K_\mu)=K_{-\mu},\quad \omega(E_i)=-F_i,\quad \omega(F_i)=-E_i \quad \forall\; \mu\in \scE\; \text{ and }\; i\in \mbI. 
\end{equation}
In this article, we shall refer to $\omega$ as the \textit{Chevalley involution} on $U_q^\scE(\mfg)$. 
It satisfies $\omega(U_q(\mfn^\pm)_\beta)=U_q(\mfn^\mp)_{-\beta}$ for all $\beta \in \msQ$, and it provides a Hopf algebra isomorphism $U_q^\scE(\mfg)\iso U_q^\scE(\mfg)^{\mathrm{cop}}$, where $U_q^\scE(\mfg)^{\mathrm{cop}}$ is the co-opposite Hopf algebra to $U_q^\scE(\mfg)$.

%
The second involution of interest to us is the anti-automorphism $\uptau$ of $U_q^\scE(\mfg)$ uniquely determined by the formulas 
\begin{equation}\label{uptau}
\uptau(K_\mu)=K_\mu, \quad \uptau(E_i)=F_i K_i, \quad \uptau(F_i)=K_i^{-1}E_i \quad \forall\; \mu\in \scE\; \text{ and }\; i\in \mbI. 
\end{equation}
It provides a Hopf algebra isomorphism $U_q^\scE(\mfg)\iso U_q^\scE(\mfg)^{\mathrm{op}}$, where $U_q^\scE(\mfg)^{\mathrm{op}}$ is the opposite Hopf algebra to $U_q^\scE(\mfg)$,
and has the property that 
\begin{equation*}
\uptau(U_q(\mfn^+)_\beta)K_\beta^{-1}=U_q(\mfn^-)_{-\beta}  \quad \text{ and }\quad K_\beta \uptau(U_q(\mfn^-)_{-\beta})=U_q(\mfn^+)_\beta
\end{equation*}
for each $\beta\in \msQ_+$. 
\subsection{The representations \texorpdfstring{$V$}{V} and  \texorpdfstring{$\End(V)$}{End(V)}}
\label{ssec:Uqg-reps}

For the remainder of Section \ref{sec:Prelim} we shall narrow our focus to the Drinfeld--Jimbo algebra $U_q^\msQ(\mfg)$, which will henceforth be denoted simply by $U_q(\mfg)$. 

Recall that a $U_q(\mfg)$-module $\EuScript{V}$ is said to be of \textit{type $1$} if it admits a weight space decomposition of the form  $\EuScript{V}=\bigoplus_{\mu \in \wlattice} \EuScript{V}_\mu$ 
with 
\begin{equation*}
\EuScript{V}_\mu=\{v\in \EuScript{V}: K_\beta v= q^{(\mu,\beta)}v \quad \forall \; \beta \in\msQ\}.
\end{equation*}
The category of finite-dimensional type $1$ representations of $U_q(\mfg)$ is semisimple, with the simple modules labeled by dominant integral weights via a highest weight theory. In particular, for any $\gamma\in \wlattice_+$ there is a unique, up to isomorphism, finite-dimensional highest weight module $V(\gamma)$ of type $1$ with the highest weight $\gamma$.  That is, $V(\gamma)$ is generated by a nonzero vector $v\in V(\gamma)$ with the property that 
\begin{equation*}
K_i v=q^{(\gamma,\alpha_i)} v \quad \text{ and }\quad E_i v=0 \quad \forall \quad i\in \mbI. 
\end{equation*}
The representation $V(\gamma)$ is irreducible, and every finite-dimensional irreducible type $1$ representation is isomorphic to a module of this form. 

As we shall be solely interested in the single nonzero dominant integral weight $\lambda\in \wlattice_+$ fixed in Section \ref{ssec:g}, we shall set 
\begin{equation*}
V:=V(\lambda)
\end{equation*}
 and let
$
\pi:U_q(\mfg)\to \End(V)
$
denote the associated $\Q(q)$-algebra homomorphism. In addition, we will  denote the set of weights of $V$ by $\wlattice_\lambda$:
\begin{equation*}
\wlattice_\lambda:=\{\mu\in \wlattice: V_\mu \neq 0\}\subset \wlattice.  
\end{equation*}

The endomophism space $\End(V)$ can itself be equipped with the structure of a $U_q(\mfg)$-module in the standard Hopf-theoretic way. Namely, the $U_q(\mfg)$-action is determined by the formulae
\begin{gather*}
K_i\cdot X=\pi(K_i)X\pi(K_i)^{-1},\\
E_i\cdot X=[\pi(E_i),X] \pi(K_i)^{-1} \quad \text{and}\quad F_i \cdot X= \pi(F_i)X-\pi(K_i)^{-1}X\pi(K_i)\pi(F_i)
\end{gather*}
for all $i\in \mbI$ and $X\in \End(V)$. This is a type $1$ finite-dimensional representation with set of weights laying in $\msQ$:
\begin{equation*}
\End(V)=\bigoplus_{\beta \in \msQ}\End(V)_\beta.
\end{equation*}
Moreover, this weight space decomposition is compatible with the natural $\Q(q)$-algebra grading on $\End(V)$, in that 
\begin{equation*}
\End(V)_\beta=\{X\in \End(V): X(V_\mu)\subset V_{\mu+\beta} \; \forall\; \mu\in \mfh^\ast\}=\bigoplus_{\mu\in \mfh^\ast}\Hom(V_\mu,V_{\mu+\beta}).
\end{equation*}

Let $t:\End(V)\to \End(V)$ denote the transpose operator with respect to a fixed choice of weight basis of $V$. This is a $\Q(q)$-algebra anti-automorphism of $\End(V)$ satisfying 
\begin{equation}\label{t-reverse}
t(\End(V)_\beta)=\End(V)_{-\beta} \quad \forall \; \beta\in \msQ.
\end{equation}

Let $\GLgr=\prod_{\mu}\mathrm{GL}(V_\mu)$ denote the subgroup of the general linear group $\mathrm{GL}(V)$ on the $\Q(q)$-vector space $V$ consisting of invertible matrices which preserve each weight space of $V$. 
The following lemma provides a compatibility condition between $t$ and the anti-automorphism $\uptau$ defined in \eqref{uptau}. 
\begin{lemma}\label{L:A-def}
There is an invertible matrix $A\in \GLgr$ which is the identity on $V_\lambda$, is symmetric, and satisfies 
\begin{equation*}
A^{-1}\pi(x)^t A=\pi(\uptau(x)) \quad \forall \quad x\in U_q(\mfg).
\end{equation*}
\end{lemma}
\begin{proof} Let $v\in V_\lambda$ be a nonzero vector (\textit{i.e.,} a highest weight vector). 
Consider the algebra homomorphism
\begin{equation*}
\pi^\sharp:=t\circ \pi \circ \uptau: U_q(\mfg)\to \End(V),
\end{equation*}
and let $V^\sharp$ denote the associated $U_q(\mfg)$-module structure on $V$. Then, by \eqref{uptau} and \eqref{t-reverse}, $v$ is a also a highest weight vector of weight $\lambda$ in $V^\sharp$. Indeed, we clearly have $\pi^\sharp(K_i)v=q^{(\alpha_i,\lambda)}v$ for all $i\in \mbI$, so it suffices to show that $\pi^\sharp(E_i)v=0$ for all $i\in \mbI$. By definition of $\pi^\sharp$, we have 
\begin{equation*}
\pi^\sharp(E_i)v
=
\pi(F_iK_i)^tv=\pi(K_i) (\pi(F_i)^t v).
\end{equation*}
As $\pi(F_i)\in \End(V)_{-\alpha_i}$, the property \eqref{t-reverse} implies that $\pi(F_i)^t v \in V_{\lambda+\alpha_i}=0$, as desired. 
It now follows immediately that $V$ and $V^\sharp$ are isomorphic $U_q(\mfg)$-modules. Let 
$
A:V\iso V^\sharp
$
be an isomorphism satisfying $A(v)=v$. Since $A(\pi(x)w)=\pi^\sharp(x) (A\cdot w)=\pi(\uptau(x))^t Aw $ for all $w\in V$, $A$ satisfies 
\begin{equation*}
A\pi(x)A^{-1}=\pi(\uptau(x))^t \quad \forall \; x\in U_q(\mfg). 
\end{equation*}
As $\uptau$ is an involution, this is equivalent to the main identity of the lemma. To complete the proof of the lemma, it remains to see that $A$ is symmetric. To this end, note that $A^t$ satisfies 
\begin{equation*}
A^t(\pi(x)w)=\left( \pi(x)^t A\right)^t w=\left(A\pi(\uptau(x))\right)^tw=\pi(\uptau(x))^t A^tw
\end{equation*}
for all $w\in V$ and $x\in U_q(\mfg)$, and is therefore an isomorphism of modules $V\to V^\sharp$. Since $V$ and $V^\sharp$ are simple and $A$ and $A^t$ both fix the highest weight vector $v$, we can conclude that $A=A^t$. 
\end{proof}

Next, let $\langle\,,\,\rangle=\mathsf{tr}\circ m:\End(V)\otimes \End(V)\to \Q(q)$ denote the trace form on $\End(V)$, where $m$ is the usual multiplication on $\End(V)$, given by composition. This is a non-degenerate, symmetric bilinear form. In addition, if we view $\Q(q)$ as a $\msQ$-graded vector space, concentrated in degree  zero, then $\langle\,,\, \rangle$ is $\msQ$-graded:
\begin{equation*}
\langle \End(V)_\alpha,\End(V)_\beta\rangle \subset \Q(q)_{\alpha+\beta} \quad \forall \; \alpha,\beta\in \msQ.
\end{equation*}
In addition, since $\langle\,,\rangle \circ (t\otimes t) =\langle\,,\rangle$, Lemma \ref{L:A-def} implies that
\begin{equation*}
\langle \pi(x),\pi(y)\rangle =\langle \pi(\uptau(x)),\pi(\uptau(y))\rangle \quad \forall\; x,y\in U_q(\mfg). 
\end{equation*}
%

\subsection{\texorpdfstring{$R$}{R}-matrices}
\label{ssec:Uqg-R}
We now turn towards recalling the construction and basic properties of the universal quasi $R$-matrix $\mathscr{R}^+$ of $U_q(\mfg)$ and how it gives rise to a genuine $R$-matrix $\mrR\in \End(V)^{\otimes 2}$.

To begin, we recall that there is an algebra automorphism $\Uppsi$ of $U_q(\mfg)^{\otimes 2}$ uniquely determined by the formulas
\begin{gather*}
\Uppsi(K_i\otimes 1)=K_i\otimes 1, \quad \Uppsi(1\otimes K_i)=1\otimes K_i,\\
\Uppsi(X_i^\pm\otimes 1)=X_i^\pm \otimes K_i^{\mp 1} \quad \text{ and }\quad \Uppsi(1\otimes X_i^\pm)=K_i^{\mp 1}\otimes  X_i^\pm
\end{gather*}
for all $i\in \mbI$, where $X_i^+=E_i$ and $X_i^-=F_i$. Informally, it is just given by $\Uppsi(x)=\mathrm{Ad}(q^{-\Omega_\mfh})$ where $\Omega_\mfh\in \mfh\otimes \mfh$ is the canonical tensor associated to the restriction of the Killing form to $\mfh\otimes \mfh$. This automorphism has the property that 
\begin{equation}\label{D-Psi}
\pi^{\otimes 2}(\Uppsi(x))=\Rd^{-1} (\pi^{\otimes 2}(x))\Rd  \quad \forall \; x\in U_q(\mfg)^{\otimes 2}
\end{equation}
where $\Rd$ is the diagonal operator 
\begin{equation*}
\Rd:=\sum_{\mu,\gamma\in \wlattice_\lambda} q^{(\mu,\gamma)-(\lambda,\lambda)}\mathrm{Id}_{V_\mu\otimes V_\gamma} \in \mathrm{GL}^\mfh(V^{\otimes 2})\subset \End(V^{\otimes 2}).
\end{equation*}
Here we note that, since $\wlattice_\lambda\subset\lambda+\msQ$, the scalar $(\mu,\gamma)-(\lambda,\lambda)$ belongs to $\Z$. 

With the above preliminaries at our disposal, we are now in a position to recall the definition of $\mathscr{R}^+$. 
By \cite{Lusztig-Book}*{Thm.~4.1.2}, there is a unique family of elements
\begin{equation*}
\mathscr{R}_\nu^+\in U_q(\mfn^+)_\nu\otimes U_q(\mfn^-)_{-\nu},
\end{equation*}
where $\nu$ takes values in $\msQ_+$,
satisfying $\mathscr{R}_0^+=1$ in addition to 
\begin{equation*}
\sum_{\nu\in \msQ_+^\beta}(\mathscr{R}_\nu^+ \Delta(x)_{\beta-\nu}-\Delta_\Uppsi^{\mathrm{op}}(x)_{\beta-\nu}\mathscr{R}_\nu^+ )=0=\sum_{\nu\in \msQ_+^\beta} (\mathscr{R}_\nu^+ \Delta(y)^{\nu-\beta}-\Delta_\Uppsi^{\mathrm{op}}(y)^{\nu-\beta}\mathscr{R}_\nu^+ )
\end{equation*}
for each $\beta\in \msQ_+$, $x\in U_q(\mfb^+)$ and $y\in U_q(\mfb^-)$. 
Here we have set $\Delta_\Uppsi:=\Uppsi\circ \Delta$ and made use of the following general notation: For a $\msQ$-graded vector space $\mathscr{U}$ and an element $u\in \mathscr{U}^{\otimes 2}$, we define 
\begin{equation*}
u_\gamma:=(\mathbf{1}_\gamma\otimes\mathrm{Id})(u)\quad\text{ and }\quad u^\gamma:=(\mathrm{Id}\otimes \mathbf{1}_\gamma)(u),
\end{equation*}
 where $\mathbf{1}_\gamma:\mathscr{U}\onto \mathscr{U}_\gamma$ is the idempotent associated to $\gamma$.
 Though it will not be important for us to have access to an explicit formula for $\mathscr{R}^+_\nu$ for arbitrary $\nu\in \msQ_+$, we shall frequently apply the fact that
\begin{equation*}
\mathscr{R}_{\alpha_i}^+=(q_i-q_i^{-1})E_i\otimes F_i \quad \forall \; i\in \mbI.
\end{equation*}

The defining identities of $\mathscr{R}_\nu^+$ can be expressed more compactly as
\begin{equation}\label{R^+}
\mathscr{R}^+ \Delta(x)=\Delta_\Uppsi^{\mathrm{op}}(x)\,\mathscr{R}^+ \quad \forall \quad x\in U_q(\mfg),
\end{equation}
in a suitable completion of $U_q(\mfg)^{\otimes 2}$ containing  $\mathscr{R}^+:=\sum_{\nu\in \msQ_+}\mathscr{R}_\nu^+$. We refer the reader to \cite{Lusztig-Book}*{\S4.1.1} and \cite{CPBook}*{\S10.1.D} for a more detailed discussion of the relevant completion of $U_q(\mfg)^{\otimes 2}$ while noting that the element $\mathscr{R}^+$ is  related to Lusztig's \textit{quasi $R$-matrix} $\Theta$ by $\mathscr{R}^+:=\Theta_{21}^{-1}$.  

Crucially, if $\EuScript{V}$ is any finite-dimensional $U_q(\mfg)$-module and $\pi_\EuScript{V}:U_q(\mfg)\to \End(\EuScript{V})$ is the associated homomorphism, then $(\pi_\EuScript{V}\otimes \mathrm{Id})(\mathscr{R}^+_\nu)$ and $(\mathrm{Id}\otimes \pi_\EuScript{V})(\mathscr{R}^+_\nu)$ are zero whenever $\nu$ is not one of the finitely many weights of the $U_q(\mfg)$-module $\End(\EuScript{V})$. Thus, the evaluations $(\pi_\EuScript{V}\otimes \mathrm{Id})(\mathscr{R}^+)$ and $(\mathrm{Id}\otimes \pi_\EuScript{V})(\mathscr{R}^+)$ are well-defined elements of $\End(\EuScript{V})\otimes U_q(\mfg)$ and 
$U_q(\mfg)\otimes \End(\EuScript{V})$, respectively.  We set 
\begin{equation}\label{R:def}
\mrR:= \Rd \cdot \pi^{\otimes 2}(\mathscr{R}^+)\in \End(V \otimes V).
\end{equation}
By \eqref{D-Psi} and \eqref{R^+}, this element satisfies the relation 
\begin{equation*}
\mrR\, \pi^{\otimes 2}(\Delta(x))=\pi^{\otimes 2}(\Delta^{\mathrm{op}}(x))\, \mrR
\quad \forall \quad x\in U_q(\mfg).
\end{equation*}
Moreover, $\mrR$ is an $R$-matrix: it satisfies the quantum Yang--Baxter equation 
\begin{equation*}
\mrR_{12}\mrR_{13}\mrR_{23}=\mrR_{23}\mrR_{13}\mrR_{12}\quad \text{ in }\quad \End(V)^{\otimes 3}\cong \End(V^{\otimes 3}).
\end{equation*}
This is a consequence of \eqref{D-Psi}, \eqref{R^+} and that  $\mathscr{R}^+$ satisfies the coproduct identities
\begin{equation}\label{R^+:Hopf}
\begin{aligned}
(\Delta\otimes \mathrm{Id})(\mathscr{R}^+)&=(\mathrm{Id}\otimes \Uppsi)\left(\mathscr{R}^+_{13}\right)  \cdot \mathscr{R}^+_{23},\\
(\mathrm{Id}\otimes \Delta)(\mathscr{R}^+)&=(\Uppsi\otimes \mathrm{Id})\left(\mathscr{R}^+_{13}\right)  \cdot \mathscr{R}^+_{12},
\end{aligned}
\end{equation}
where we have used the standard double subscript notation to indicate how to embed elements of the tensor square of a $\Q(q)$-vector space $\mathscr{U}$ into $\mathscr{U}^{\otimes n}$ for $n\geq 2$.

To conclude this subsection, we explain how $\mrR$ transforms when the transpose $t$ from Section \ref{ssec:Uqg-reps} is applied to both of its tensor factors. This will be applied in Sections \ref{sec:URg} and \ref{sec:URg=Uq}. 
\begin{lemma}\label{L:R^t}
Let  $A\in\GLgr$ be as in Lemma \ref{L:A-def}. Then $\mrR$ satisfies
\begin{equation*}
\mrR^{t_1,t_2}=\mathrm{Ad}(A\otimes A)\left(\mrR_{21}\right),
\end{equation*}
where $\mrR^{t_1,t_2}=(t\otimes t)(\mrR)$. 
\end{lemma}
\begin{proof}
We first note that $(\uptau\otimes \uptau)(\mathscr{R}^+_\nu)=\left(\Uppsi^{-1}(\mathscr{R}_{21}^+)\right)_\nu$ for each $\nu\in \msQ_+$, where $\uptau$ is the anti-automorphism defined in \eqref{uptau}. Equivalently, one has
\begin{equation}\label{R+:uptau}
(\uptau\otimes \uptau)(\mathscr{R}^+):=\Uppsi^{-1}(\mathscr{R}^+_{21}).
\end{equation}
This is essentially a consequence of  Proposition 4.2 of \cite{Dr-almost}, where it was shown that the universal $R$-matrix $\mathscr{R}_{U_\hbar\mfg}$ of $U_\hbar\mfg$ satisfies $(\mathscr{R}_{U_\hbar\mfg})_{21}=\uptau^{\otimes 2}(\mathscr{R}_{U_\hbar\mfg})$. To prove this rigorously in the present setting, where we do not have access to $\mathscr{R}_{U_\hbar\mfg}$, one first observes that
$\Psi(\uptau^{\otimes 2}(\mathscr{R}^+_{21}))$ satisfies the equation \eqref{R^+} and, by the definitions of $\Uppsi$ and $\uptau$, is such that $\Psi(\uptau^{\otimes 2}(\mathscr{R}^+_{21}))_0=1$ and 
\begin{equation*}
\Psi(\uptau^{\otimes 2}(\mathscr{R}^+_{21}))_\nu \in U_q(\mfn^+)_\nu\otimes U_q(\mfn^-)_{-\nu} \quad \forall \; \nu\in \msQ_+.
\end{equation*}
It follows by uniqueness that \eqref{R+:uptau} holds. Combining this with \eqref{D-Psi}, we obtain 
\begin{equation} \label{pi-R-tau}
 (\pi\circ \uptau)^{\otimes 2}(\mathscr{R}^+)=\Rd \left(\pi^{\otimes 2}(\mathscr{R}_{21}^+)\right)\Rd^{-1}.
\end{equation}
The assertion of the lemma now follows from the relation $A^{-1}\pi(x)^t A=\pi(\uptau(x))$ established in Lemma \ref{L:A-def}. Indeed, we have 
\begin{align*}
\mrR^{t_1,t_2}=\pi^{\otimes 2}(\mathscr{R}^+)^{t_1,t_2}\Rd&=A_1A_2 (\pi\circ\uptau)^{\otimes 2}(\mathscr{R}^+) A_2^{-1} A_1^{-1} \Rd\\
&
=
A_1A_2 \Rd \pi^{\otimes 2}(\mathscr{R}^+_{21})\Rd^{-1} A_2^{-1} A_1^{-1} \Rd\\
&
=A_1A_2 \mathrm{R}_{21} \Rd^{-1} A_2^{-1} A_1^{-1} \Rd, 
\end{align*}
where $A_1=A\otimes I$ and $A_2=I\otimes A$. 
Since $A_1A_2$ preserves weight spaces, it commutes with $\Rd$, so the right-hand side above is really just $A_1A_2 \mrR_{21} A_2^{-1} A_1^{-1}$, as desired. \qedhere
\end{proof}

\section{The quantum enveloping algebra \texorpdfstring{$\mathbf{U}_q^\lambda(\mfg)$}{U\_q\^{}lambda(g)}}\label{sec:Uqlambda}

As indicated in Section \ref{ssec:results}, there exists a quantum group $\bUqg$ which contains $U_q(\mfg)$ as a subalgebra, admits the Chevalley--Serre presentation spelled out in Theorem \ref{T:intro}, and which can be naturally reconstructed from the solution \eqref{R:def} of the quantum Yang--Baxter equation. In this section, we introduce $\bUqg$ starting from the quantized enveloping algebra $U_q^\EuScript{E}(\mfg)$ associated to the lattice $\EuScript{E}=\wlattice(\lambda)$, defined in Section \ref{ssec:Lattice} below, which is generated by the set of weights of $V$.
\subsection{The lattice associated to a dominant integral weight}
\label{ssec:Lattice}

Recall that $\lambda$ is a fixed nonzero dominant integral weight of $\mfg$, and that $V$ is the associated finite-dimensional irreducible $U_q(\mfg)$-module with highest weight $\lambda$. 
Let $\wlattice(\lambda)$ denote the $\Z$-submodule of $\wlattice$ generated by the set of weights $\wlattice_\lambda$: 
\begin{equation*}
\wlattice(\lambda):=\Z \wlattice_\lambda= \sum_{\mu \in \wlattice_\lambda}\Z \mu.
\end{equation*}
The following elementary lemma establishes some basic properties of $\wlattice(\lambda)$.
\begin{lemma}\label{L:lattice-V}
For each $i\in \mbI$, there exists  $\mu,\gamma\in \wlattice_\lambda$ such that $\alpha_i=\mu-\gamma$. Consequently, one has $\msQ\subset \wlattice(\lambda)\subset \wlattice$. Moreover, the image $[\mu]$ of any weight $\mu$ in $\wlattice(\lambda)/\msQ$ coincides with $[\lambda]$ and generates the group $\wlattice(\lambda)/\msQ$. In particular, one has
\begin{equation*}
\wlattice(\lambda)/\msQ\cong \Z/\olambda \Z.
\end{equation*}
\end{lemma}
\begin{proof}
Let $\mathbb{V}(\lambda)$ denote the finite-dimensional irreducible representation of the complex simple Lie algebra $\mfg$ with the highest weight $\lambda$. Since $\lambda$ is nonzero and $\mfg$ is simple, this is a faithful representation. Moreover, the set of weights  of $\mathbb{V}(\lambda)$ is identical to the set of weights of $V(\lambda)$, with $\dim_\C \mathbb{V}(\lambda)_\gamma=\dim_{\Q(q)}V(\lambda)_\gamma$ for any weight $\gamma$; see \cite{Lusztig-Book}*{Thm.~33.1.3} or \cite{Jantzen-Book}*{Thm.~5.15}, for instance. The assertion of the lemma now follows from a standard exercise in Lie theory; see Exercise 5 of \cite{Humphreys-book}*{\S21.4}, for instance. As the argument is brief, we include it for the sake of completeness.

Suppose towards a contradiction that there does not exist a pair $\mu,\gamma\in \wlattice_\lambda$ as in the statement of the lemma. Since the $\alpha_i$-root space $ \mfg_{\alpha_i}$ satisfies $\mfg_{\alpha_i}\cdot \mathbb{V}(\lambda)_\gamma\subset \mathbb{V}(\lambda)_{\gamma+\alpha_i}$ for all $\gamma\in \wlattice_\lambda$ and $\gamma$ and $\gamma+\alpha_i$ are not both weights, $\mfg_{\alpha_i}$ annihilates all of $\mathbb{V}(\lambda)$. Since $\mathbb{V}(\lambda)$ is faithful, this is impossible. 

The above shows that $\msQ\subset \wlattice(\lambda)$. The remaining statements of the lemma now follow from the fact that $\wlattice(\lambda)=\Z\wlattice_\lambda$ and  $\wlattice_\lambda\subset \lambda+\msQ$. \qedhere 
\end{proof}
%
%

\subsection{The quantum algebra \texorpdfstring{$\mathbf{U}_q^\lambda(\mfg)$}{U\_q\^{}lambda(g)}}
\label{ssec:bUqg}

Henceforth, we shall write $U_q^\lambda(\mfg)$ for the Hopf algebra $U_q^{\scE}(\mfg)$ associated to the lattice $\scE=\wlattice(\lambda)$ from Section \ref{ssec:Lattice}. We are now prepared to introduce the algebraic structure at the heart of Section \ref{sec:Uqlambda}. 
\begin{definition}\label{D:bUqg}
We define $\bUqg$ to be the subalgebra of the $\Q(q)$-algebra 
\begin{equation*}
U_q^{\lambda}(\mfg)\otimes \Q(q)[v^{\pm 1}]
\end{equation*}
 generated by $\{E_i,F_i\}_{i\in \mbI}$ and $\{K_\mu^{\pm 1} v^{\mp 1}\}_{\mu\in \wlattice_{\lambda}}$.
\end{definition}
There is a unique $\msQ$-graded Hopf algebra structure on $U_q^\lambda(\mfg)\otimes \Q(q)[v^{\pm 1}]$ for which $v$ is a grouplike element of degree zero and the natural embedding $U_q^\lambda(\mfg)\into U_q^\lambda(\mfg)\otimes \Q(q)[v^{\pm 1}]$ is a morphism of graded Hopf algebras. Since for each weight $\mu\in \wlattice_\lambda$ of $V$ one has
\begin{equation*}
\Delta(K_\mu v^{-1})=K_\mu  v^{-1} \otimes K_\mu  v^{-1}, \quad \veps(K_\mu v^{-1})=1, \quad S(K_\mu v^{-1})=v K_\mu^{-1},
\end{equation*}
$\bUqg$ is a $\msQ$-graded Hopf subalgebra of $U_q^\lambda(\mfg)\otimes \Q(q)[v^{\pm 1}]$. In particular, one has the decomposition
\begin{equation*}
\bUqg=\bigoplus_{\beta \in \msQ} \bUqg_\beta,
\end{equation*}
where $\bUqg_\beta=\bUqg\cap( U_q^\lambda(\mfg)_\beta\otimes \Q(q)[v^{\pm 1}])$ for each $\beta\in \msQ$.

The following result shows that $\bUqg$ can be presented in terms of generators and relations by adding a single invertible element to the realization of $U_q(\mfg)$ provided by Definition \ref{D:Uqg}.
\begin{proposition}\label{P:gen-rels}
$\bUqg$ is isomorphic to the unital, associative $\Q(q)$-algebra $\dot\mbU_q^\lambda(\mfg)$  generated by $\{\upxi_i^\pm,x_i^\pm\}_{i\in \mbI}\cup\{\upxi_\lambda^\pm\}$, subject to the following list of relations:
\begin{gather}
\upxi_i^\pm\upxi_i^\mp=1=\upxi_\lambda^\pm\upxi_\lambda^\mp, \quad[\upxi_i^\pm, \upxi_j^\pm]=0=[\upxi_i^\pm,\upxi_\lambda^\pm],
\label{dot:1}\\
\upxi_i^+ x_j^\pm \upxi_i^-=q^{\pm (\alpha_i,\alpha_j)}x_j^\pm, \quad \upxi_\lambda^+ x_j^\pm \upxi_\lambda^-=q^{\pm (\lambda,\alpha_j)}x_j^\pm,
\label{dot:2}\\
[x_i^+,x_j^-]=\delta_{ij}\frac{\upxi_i^+-\upxi_i^-}{q_i-q_i^{-1}},
\label{dot:3}\\
\sum_{b=0}^{1-a_{ij}}(-1)^b \sbinom{1-a_{ij}}{b}_{q_i} (x_i^\pm)^b x_j^\pm (x_i^\pm)^{1-a_{ij}-b}=0,
\label{dot:4}
\end{gather}
where $q_i=q^{d_i}$ and $i\neq j$ in the last relation.  Explicitly, the assignment
\begin{equation*}
x_i^+\mapsto E_i, \quad x_i^-\mapsto F_i,\quad  \upxi_i^\pm \mapsto K_i^{\pm 1},\quad  \upxi_\lambda^\pm \mapsto K_\lambda^{\pm 1} v^{\mp 1}
\end{equation*}
uniquely extends to an isomorphism of $\Q(q)$-algebras $\varphi:\dot\mbU_q^\lambda(\mfg)\iso \mbU_q^\lambda(\mfg)$. 
\end{proposition}

\begin{proof}
It follows readily from the definition of $\bUqg$ that the given assignment extends to a homomorphism $\varphi:\dot\mbU_q^\lambda(\mfg)\to \mbU_q^\lambda(\mfg)$, which is surjective since the lattice $\wlattice(\lambda)$ is generated as a $\Z$-module by the simple roots $\{\alpha_i\}_{i\in \mbI}$ and the highest weight $\lambda$.

Let's now construct the inverse of $\varphi$ explicitly. Since $\olambda\lambda\in \msQ$, we can write $\olambda\lambda=\sum_{i\in \mbI}n_i \alpha_i$ with $n_i\in \Z$ for each $i\in \mbI$.  Set 
\begin{equation}\label{y:1}
y:=(\upxi_\lambda^-)^{\olambda}\prod_{j\in \mbI} (\upxi_j^+)^{n_j}\in \dot\mbU_q^\lambda(\mfg).
\end{equation}
It follows from \eqref{dot:1} and \eqref{dot:2} that this is an invertible central element. Consider now the $\Q(q)$-algebra 
\begin{equation*}
\ddot\mbU_q^\lambda(\mfg):=(\dot\mbU_q^\lambda(\mfg)\otimes \Q(q)[v^{\pm 1}]) /(y-v_\lambda)
\end{equation*}
where $v_\lambda=v^{\olambda}\in \Q(q)[v^{\pm 1}]$. There is a natural algebra homomorphism 
\begin{equation*}
\imath:\dot\mbU_q^\lambda(\mfg)\to \ddot\mbU_q^\lambda(\mfg)
\end{equation*}
induced by the inclusion $\dot\mbU_q^\lambda(\mfg)\into \dot\mbU_q^\lambda(\mfg)\otimes \Q(q)[v^{\pm 1}]$. Note that $\imath$ is injective. Indeed, if $x\in \dot\mbU_q^\lambda(\mfg)$ satisfies $x=z \cdot (y-v_\lambda)$ for some $z\in \dot\mbU_q^\lambda(\mfg)\otimes \Q(q)[v^{\pm 1}]$, then writing $z=\sum_{n\in \Z} z_n v^n$ (with only finitely many $z_n$ nonzero), we see that 
\begin{equation*}
z \cdot (y-v_\lambda)=\sum_{n\in \Z} z_n y v^n -\sum_{n\in \Z} z_n v^{n+\olambda}=
\sum_{n\in \Z}(z_n y-z_{n-\olambda})v^n
\end{equation*}
must have degree $0$ in $v$. In particular, $z_n y=z_{n-\olambda}$ for all $n\neq 0$ and $x=z_0y-z_{-\olambda}$. But then $z_{-\olambda}=z_{-2\olambda}y^{-1}=z_{-3\olambda}y^{-2}=\ldots=z_{-k\olambda}y^{-k+1}$ for any $k>0$ (so $z_{-\olambda}=0$). 
Similarly, 
\begin{equation*}
z_{0}=z_{\olambda-\olambda}=z_{\olambda}y=z_{2\olambda}y^2=\cdots=z_{k\olambda}y^{k}
\end{equation*}
for each $k\geq 0$, so we must have $z_0=0$. Thus, $x=z_0y-z_{-\olambda}=0$. This shows that $\imath$ is injective.

Next, we construct a homomorphism $U_q^\lambda(\mfg)\to \ddot\mbU_q^\lambda(\mfg)$ as follows. For each $\gamma\in \wlattice(\lambda)=\msQ+\Z\lambda$ and decomposition $\gamma=\sum_{i\in \mbI} c_i \alpha_i + m\lambda$, set 
\begin{equation*}
\dot\upxi_\gamma^+:=(\upxi_\lambda^+)^m v^{m}\prod_{j\in \mbI}(\upxi_j^+)^{c_j}\in \ddot\mbU_q^\lambda(\mfg).
\end{equation*}
\noindent\textit{Claim}. The element $\dot\upxi_\gamma^+$ does not depend on the choice of decomposition $\gamma=\sum_{i\in \mbI} c_i \alpha_i + m\lambda$.

\begin{proof}[Proof of claim] \let\qed\relax
Suppose $\gamma=\sum_{i\in \mbI} b_i \alpha_i + n\lambda$ is another such a decomposition. We must show that in $\ddot\mbU_q^\lambda(\mfg)$, one has 
\begin{equation*}
(\upxi_\lambda^+)^{m-n} v^{m-n}\prod_{j\in \mbI}(\upxi_j^+)^{c_j-b_j}=1.
\end{equation*}
Since $(n-m)\lambda=\sum_{j\in \mbI}(b_j-c_j)\alpha_j\in \msQ$, we have $a\olambda=n-m$ for some $a\in \Z$, and thus $(n-m)\lambda=a\olambda\lambda=\sum_{i\in \mbI} an_i \alpha_i$. Therefore,
$b_j-c_j=an_j$ for all $j\in \mbI$ and 
\begin{equation*}
(\upxi_\lambda^+)^{m-n} v^{m-n}\prod_{j\in \mbI}(\upxi_j^+)^{c_j-b_j}
=
(\upxi_\lambda^+)^{-a\olambda} v^{-a\olambda}\prod_{j\in \mbI}(\upxi_j^+)^{-an_j}=v^{a\olambda}v^{-a\olambda}=1. 
\end{equation*}
\end{proof}
Given the claim, it follows readily that there is an algebra homomorphism $\phi:U_q^\lambda(\mfg)\to \ddot\mbU_q^\lambda(\mfg)$ uniquely determined by 
\begin{equation*}
K_\gamma^{\pm 1} \mapsto \dot\upxi_\gamma^\pm, \quad E_i\mapsto x_i^+, \quad F_i \mapsto x_i^- \quad \forall \quad \gamma\in \wlattice(\lambda)\; \text{ and }\;  i\in \mbI,
\end{equation*}
where $\dot\upxi_\gamma^-=(\dot\upxi_\gamma^+)^{-1}$.  Since there is also a natural algebra homomorphism from $\Q(q)[v^{\pm 1}]$ to the center of $\ddot\mbU_q^\lambda(\mfg)$ sending $v$ to $v$, we obtain a homomorphism
\begin{equation*}
\Phi:U_q^\lambda(\mfg)\otimes \Q(q)[v^{\pm 1}]\to \ddot\mbU_q^\lambda(\mfg)
\end{equation*}
which sends a simple tensor $x\otimes p(v)$ to $\phi(x)p(v)$ for any Laurent polynomial $p(v)$ with coefficients in $\Q(q)$. Note that 
\begin{equation*}
\Phi(K_\lambda^{\pm 1}v^{\mp 1})=\dot\upxi_\lambda^{\pm} v^{\mp 1}=\upxi_\lambda^\pm v^{\pm 1}v^{\mp 1}=\upxi_\lambda^\pm \quad \text{ and }\quad \Phi(K_i^{\pm 1})=\xi_i^\pm \quad \forall\; i\in \mbI.
\end{equation*} 
It follows  that $\Phi$ restricts to a surjection $\mbU_q^\lambda(\mfg)\to \imath(\dot\mbU_q^\lambda(\mfg))\cong \dot\mbU_q^\lambda(\mfg)$, which is necessarily $\varphi^{-1}$. 
\end{proof}
%
%

\subsection{The algebraic structure of \texorpdfstring{$\mathbf{U}_q^\lambda(\mfg)$}{U\_q\^{}lambda(g)}}
\label{ssec:bUqg-alg}

In this subsection, we highlight a number of additional structural properties for $\bUqg$ which follow from its definition and Proposition \ref{P:gen-rels}.

Given $\zeta\in \Q(q)^\times$, let $\dot\upchi_{q,\zeta}$ be the algebra automorphism of $\Q(q)[v^{\pm 1}]$ uniquely determined by $\dot\upchi_{q,\zeta}(v)=\zeta v$. 
Then the tensor product 
\begin{equation*}
\upchi_{q,\zeta}:=\mathrm{Id}_{U_q^\lambda(\mfg)} \otimes \dot\upchi_{q,\zeta}
\end{equation*}
is an algebra automorphism of $U_q^\lambda(\mfg)\otimes\Q(q)[v^{\pm 1}]$. It is easy to see that the subalgebra of $U_q^\lambda(\mfg)\otimes\Q(q)[v^{\pm 1}]$ consisting of all elements fixed by each $\upchi_{q,\zeta}$ coincides with the quantum enveloping algebra $U_q^\lambda(\mfg)$ associated to the lattice $\wlattice(\lambda)$. Moreover, each $\upchi_{q,\zeta}$ restricts to an automorphism of $\bUqg$. We set 
\begin{equation*}
\bUqg^{\upchi}:=\bigcap_{\zeta\in \Q(q)^\times} \bUqg^{\upchi_{q,\zeta}}
\end{equation*}
where $\bUqg^{\upchi_{q,\zeta}}$ consists of all $x\in \bUqg$ for each $\upchi_{q,\zeta}(x)=x$. We then have the following result.
\begin{proposition}\label{P:bUqg-str}
Let $v_\lambda:=v^{\olambda}$. Then $\bUqg$ has the following properties:

\begin{enumerate}[font=\upshape]\setlength{\parskip}{3pt}
\item\label{Uqlambda:2}  $\bUqg$ contains $U_q(\mfg)\otimes \Q(q)[v_\lambda^{\pm 1}]$ as a subalgebra, which is proper if and only if $\lambda\notin\msQ$.  Moreover, one has 
\begin{equation*}
 U_q^\lambda(\mfg)\cap \bUqg=U_q(\mfg)\quad \text{ and }\quad \Q(q)[v^{\pm 1}]\cap \bUqg=\Q(q)[v_\lambda^{\pm 1}]. 
 \end{equation*}

\item\label{Uqlambda:3} The subalgebra $\bUqg^{\upchi}$ is equal to $U_q(\mfg)$. 
\item\label{Uqlambda:4} The restriction of the epimorphism $\mathrm{Id}\otimes \veps:U_q^\lambda(\mfg)\otimes \Q(q)[v^{\pm 1}]\onto U_q^\lambda(\mfg)$ to $\bUqg$ induces an isomorphism of Hopf algebras 
\begin{equation*}
\bUqg/(v_\lambda-1)\iso U_q^\lambda(\mfg). 
\end{equation*}
\end{enumerate}
\end{proposition}
We omit the proof of this result; as indicated at the beginning of the section, it follows from Proposition \ref{P:gen-rels}, its proof, and the original  definition of $\bUqg$ as a subalgebra of $U_q^\lambda(\mfg)\otimes \Q(q)[v^{\pm 1}]$. We shall not apply it until Section \ref{ssec:Cons-R}.

\subsection{The generating matrices \texorpdfstring{$\scL^+$}{L\^{}+} and \texorpdfstring{$\scL^-$}{L\^{}-}}
\label{ssec:bUqg-gen}

The goal of this subsection is to introduce generating matrices for $\bUqg$ which satisfy the well-known $\textit{RLL}$ quadratic matrix relations and play the role of $L$-operators with respect to the underlying representation $V$ (see \cite{KS-book}*{\S8.5}, for instance).  

Let's begin by recalling some helpful notation which frequently arises in the $R$-matrix formalism for quantum groups. Suppose we are given an arbitrary unital, associative $\Q(q)$-algebra $\mathscr{U}$. For each pair of positive integers $n$ and $k$ with $k\leq n$, let $\imath_{\mathscr{U},n}^{(k)}:\mathscr{U}\to \mathscr{U}^{\otimes n}$ denote the homomorphism
\begin{equation*}
\imath_{\mathscr{U},n}^{(k)}(x)=1_\mathscr{U}^{\otimes (k-1)}\otimes x \otimes 1_\mathscr{U}^{\otimes (n-k)} \quad \forall \quad x\in \mathscr{U}.
\end{equation*}
Given an additional positive integer $m$ and $T\in \End(V)\otimes \mathscr{U}$, we set
\begin{equation*}
T_{i[j]}:=(\imath^{(i)}_{\End(V),n}\otimes \imath^{(j)}_{\mathscr{U},m})(F)\in \End(V)^{\otimes n} \otimes \mathscr{U}^{\otimes m},
\end{equation*}
where $1\leq i\leq n$ and $1\leq j\leq m$ are fixed. We will write $T_{[j]}$ for $T_{1[j]}$ and $T_{i}$ for $T_{i[1]}$ in the special cases where $n=1$ and $m=1$, respectively.

We now turn towards defining $\scL^+$ and $\scL^-$. We introduce the matrix $\scK$ by 
\begin{equation*}
\scK:=\sum_{\mu\in \wlattice_\lambda} \mathrm{Id}_{V_{\mu}}\otimes K_\mu \in \End(V)\otimes U_q^\lambda(\mfg).
\end{equation*}
Then $\scK v^{-1} \in \End(V)\otimes \bUqg$ and $\scK$ satisfies the relation 
\begin{equation}\label{Uppsi-K}
(\pi \otimes \mathrm{Id})( \Uppsi(x))=\scK^{-1} (\pi \otimes \mathrm{Id})(x)\scK \quad \forall \; x\in U_q(\mfg)^{\otimes 2},
\end{equation}
where $\Uppsi$ is as in Section \ref{ssec:Uqg-R}. 
Next, we introduce $\scL^\pm \in \End(V)\otimes \bUqg$ by setting
\begin{equation*}
\scL^+:=v\scK^{-1} \cdot (\pi\otimes \omega)(\mathscr{R}^+) \quad \text{ and }\quad 
\scL^-:=(\pi\otimes \omega)((\mathscr{R}^+_{21})^{-1})\cdot \scK v^{-1}, 
\end{equation*}
where $\omega$ is the Chevalley involution on $U_q(\mfg)\subset \bUqg$ defined in \eqref{omega}, and $\mathscr{R}^+$ is as in Section \ref{ssec:Uqg-R}.

\begin{proposition}\label{P:bUqg-L}
The matrices $\scL^+$ and $\scL^-$ have the following properties:
\begin{enumerate}[font=\upshape]
\item\label{L:1} Their coefficients generate $\bUqg$ as a $\Q(q)$-algebra. 
\item\label{L:2} They satisfy the Hopf algebraic relations 
\begin{gather*}
\Delta(\scL^\pm)=\scL^\pm_{[1]}\scL^\pm_{[2]},\quad S(\scL^\pm)=(\scL^\pm)^{-1}\quad \text{ and }\quad
\veps(\scL^\pm)=I.
\end{gather*}
\item\label{L:3} They satisfying the relations
\begin{equation*}
\mrR \scL^\pm_1 \scL^\pm_2 =\scL^\pm_2\scL^\pm_1\mrR \quad \text{ and }\quad
\mrR \scL^+_1 \scL^-_2 =\scL^-_2\scL^+_1\mrR
\end{equation*}
in the algebra $\End(V)^{\otimes 2}\otimes \bUqg$.
\end{enumerate}
\end{proposition}
\begin{proof}
Since $\mathscr{R}_0^+=1$,  $\mathscr{R}_{\alpha_i}^+=(q_i-q_i^{-1})E_i\otimes F_i$ and $\left((\mathscr{R}^+_{21})^{-1}\right)_{-\alpha_i}=-(q_i-q_i^{-1})F_i\otimes E_i$, the definitions of $\scL^+$ and $\scL^-$ yield that  
\begin{equation*}
\scL_{\alpha_i}^+= (q_i^{-1}-q_i) v\scK^{-1}\cdot \pi(E_i)\otimes E_i \quad \text{ and }\quad \scL_{-\alpha_i}^-=(q_i-q_i^{-1})\pi(F_i)\otimes F_i \cdot \scK v^{-1}
\end{equation*}
for all $i \in \mbI$. 
The assertion of Part \eqref{L:1} now follows from the fact that $\bUqg$ is generated by $\{E_i,F_i\}_{i\in \mbI}$ and the coefficients $\{K_\mu^{\pm 1}v^{\mp 1}\}_{\mu\in \wlattice_\lambda}$ of the $\End(V)_0\otimes \bUqg$ component $\scK^\pm v^{\mp 1}$ of $\scL^\pm$. 

The relations for $S$ and $\veps$ stated in \eqref{L:2} follow readily from the Hopf algebra axioms and the relation $\Delta(\scL^\pm)=\scL^\pm_{[1]}\scL^\pm_{[2]}$, which itself is deduced from \eqref{R^+:Hopf} using \eqref{Uppsi-K}. Similarly, the identities of \eqref{L:3} are deduced using \eqref{R^+}, \eqref{R^+:Hopf} and \eqref{Uppsi-K}. We refer the reader to \cite{KS-book}*{Prop.~8.5.27}, for example, for more details.  \qedhere 
\end{proof}
%
\section{The quantum algebra \texorpdfstring{$\mathbf{U}_\mathrm{R}^\lambda(\mfg)$}{U\_R\^{}lambda(g)}}\label{sec:URg}

We now shift our attention to the algebraic structure at the heart of this article: the Hopf algebra $\URg$ output by the $R$-matrix formalism for quantum groups applied to the solution \eqref{R:def} of the quantum Yang--Baxter equation.

\subsection{Definition of \texorpdfstring{$\mbU_\mrR^\lambda(\mfg)$}{U{\_R}g} and first properties}\label{ssec:URg-def}
In what follows, all notation is as in Section \ref{ssec:Uqg-reps}. 
Consider the subalgebras 
\begin{equation*}
\mathscr{E}^\pm:=\bigoplus_{\beta \in \msQ_{\pm}} \End(V)_\beta \subset \End(V).
\end{equation*}
As the trace form $\langle\,,\rangle$ is $\msQ$-graded, it restricts to a non-degenerate bilinear form $\langle \,,\, \rangle_\pm:\mathscr{E}^\pm \otimes \mathscr{E}^\mp\to \Q(q)$. Let $\Omega^\pm \in \mathscr{E}^\pm \otimes \mathscr{E}^\mp$ denote the canonical element associated to  $\langle \,,\, \rangle_\pm$, and consider the external direct sum $\mathscr{E}:=\mathscr{E}^+\oplus \mathscr{E}^-$. Define
\begin{equation}\label{mrL}
\mrL^\pm:=(1\otimes \imath^\mp)(\Omega^\pm)\in \End(V)\otimes T(\mathscr{E}),
\end{equation}
where $T(\mathscr{E})$ is the tensor algebra on the $\Q(q)$-vector space $\mathscr{E}$, and $\imath^\mp:\mathscr{E}^\mp\into T(\mathscr{E}^\mp)\subset T(\mathscr{E})$ is the natural inclusion. In what follows, we let $\Lz^\pm$ denote the zero weight component of $\mrL^\pm$: 
\begin{equation}\label{Lz}
\Lz^\pm:=\mrL_0^\pm \in \End(V)_0 \otimes T(\mathscr{E}). 
\end{equation}
\begin{definition}\label{D:URg}
We define $\mbU_\mrR^\lambda(\mfg)$ to be the quotient of the $\Q(q)$-algebra $T(\mathscr{E})$ by the two sided ideal generated by the relations
\begin{gather}
\Lz^+\Lz^-=I=\Lz^-\Lz^+, \label{Lz+Lz-}\\
\mrR \mrL_1^+ \mrL_2^-=\mrL_2^-\mrL_1^+ \mrR, \label{RLL+-}\\
\mrR \mrL_1^+\mrL_2^+=\mrL_2^+\mrL_1^+ \mrR \quad \text{ and }\quad \mrR \mrL_1^-\mrL_2^-=\mrL_2^-\mrL_1^- \mrR  \label{RLL}
\end{gather}
in $\End(V)\otimes T(\mathscr{E})$ and $\End(V)^{\otimes 2}\otimes T(\mathscr{E})$, respectively.
\end{definition}
\begin{remark}
Let us rephrase the above definition in more familiar terms. 
Suppose that $\{v_a\}_{a\in \mcJ}$ is any fixed basis of $V$, and let $\{E_{ab}\}_{a,b\in \mcJ}\subset \End(V)$ be the associated basis of matrix units in $\End(V)$, defined by $E_{ab}v_c=\delta_{bc}v_a$. 

Then $\mbU_\mrR^\lambda(\mfg)$ is isomorphic to the $\Q(q)$-algebra $\mbU_\mrR^\mcJ(\mfg)$ generated by elements $\{\mathscr{l}_{ab}^{\pm}\}_{a,b\in \mcJ}$ subject to the relations \eqref{Lz+Lz-}--\eqref{RLL}, with $\mrL^\pm$ replaced by the matrix 
$
\mrL^{\pm}_\mcJ:=\sum_{a,b}E_{ab}\otimes \mathscr{l}_{ab}^{\pm},
$
in addition to the triangularity relations 
\begin{equation}\label{triangularity}
\mrL^{\pm}_\mcJ= \sum_{\beta\in \msQ_\pm}\mrL_{\mcJ,\beta}^{\pm}.
\end{equation}
The relation \eqref{triangularity} is automatically satisfied by $\mrL^\pm$ from Definition \ref{D:URg}, and the assignment $\mrL^{\pm}_\mcJ\mapsto \mrL^\pm$ uniquely extends to an isomorphism
$\mbU_\mrR^\mcJ(\mfg)\iso \mbU_\mrR^\lambda(\mfg)$. The advantage of Definition \ref{D:URg} is that it does not depend on any choice of bases for $V$ or $\End(V)$. 
\end{remark}
\begin{remark}
Henceforth, we shall always use the notation  $\mrL^\pm$ and $\Lz^\pm$  to refer to the images of the elements \eqref{mrL} and \eqref{Lz} in $\End(V)\otimes \URg$. In the rare instances where it is necessary to refer to their preimages in $\End(V)\otimes T(\mathscr{E})$ (as will be the case in Proposition \ref{P:grading}), we will write $\tilde{\mrL}^\pm$ and $\tilde{\Lz}^\pm$.
\end{remark}

Let us now turn to describing the Hopf algebra structure on $\URg$. Since the weight zero component $\Lz^\pm$ of $\mrL^\pm$ is invertible,  $\mrL^\pm$ is itself an invertible element in 
\begin{equation*}
\bigoplus_{\beta\in \msQ_\pm} \End(V)_\beta \otimes \URg \subset \End(V)\otimes \URg.
\end{equation*}
 Given this observation, it follows by a standard argument that $\URg$ admits the structure of a Hopf algebra with coproduct $\Delta_\mrR$, antipode $S_\mrR$ and counit $\veps_\mrR$ uniquely determined by
\begin{equation}\label{URg-Hopf}
\Delta_\mrR(\mrL^\pm)=\mrL_{[1]}^\pm \mrL_{[2]}^\pm,\quad S_\mrR(\mrL^\pm)=(\mrL^\pm)^{-1} \quad \text{ and }\quad \veps_\mrR(\mrL^\pm)=I. 
\end{equation}
We refer the reader to Propositions 8.32 and 9.1 of \cite{KS-book}, for instance, in addition to Remark 5.8 of \cite{GRW-RQUE} and Theorems 1 and 9 of the foundational paper \cite{FRT} for further details.
One subtle technical point which does not arise in these sources is that, when the weight spaces of $V$ are not all one-dimensional, it is non-trivial that the assignments $\Delta_\mrR$ and $S_\mrR$  preserve the defining relation \eqref{Lz+Lz-} of $\URg$. For instance, one has 
\begin{equation*}
\Delta_\mrR(\Lz^+)\Delta_\mrR(\Lz^-)=\Lz^+_{[1]}\Lz^+_{[2]}\Lz^-_{[1]}\Lz^-_{[2]},
\end{equation*}
which, by \eqref{Lz+Lz-}, will be the unit in $\End(V)\otimes \URg^{\otimes 2}$ provided $\Lz^+_{[2]}\Lz^-_{[1]}=\Lz^-_{[1]}\Lz^+_{[2]}$. Since $\End(V)_0$ is not in general a commutative algebra, the latter relation is itself non-trivial. However, it is an immediate consequence of the first assertion of Theorem \ref{T:Lz}, which more generally provides the missing ingredient needed to see  $\Delta_\mrR$ and $S_\mrR$ preserve \eqref{Lz+Lz-}. For the moment, we take this for granted and defer the statement of Theorem \ref{T:Lz} to Section \ref{ssec:Lz-semi} in order to first establish more basic structural properties of $\URg$. We further emphasize that the proof of Theorem \ref{T:Lz} will depend only on the defining relations of $\URg$ and properties of $V$ itself. 
 
As the next result illustrates, the $\msQ$-grading on $\End(V)$ naturally induces a $\msQ$-grading on $\URg$ compatible with its Hopf algebra structure. 
\begin{proposition}\label{P:grading}

There is a unique $\msQ$-grading on the Hopf algebra $\URg$ with the property that 
\begin{equation*}
\mrL_\beta^\pm \in \End(V)_\beta\otimes \URg_\beta \quad \forall \quad \beta \in \msQ. 
\end{equation*}
\end{proposition}
\begin{proof}
Let us equip the vector spaces $\mathscr{E}^+$ and $\mathscr{E}^-$ from the beginning of Section \ref{ssec:URg-def} with their opposite $\msQ$-gradings, in which $(\mathscr{E}^\pm)_\beta=\End(V)_{-\beta}$ and $(\mathscr{E}^\pm)_{-\beta}=0$ for all $\beta\in \msQ_\pm$. Then the tensor algebra $T(\mathscr{E})$  on the $\msQ$-graded space $\mathscr{E}=\mathscr{E}^+\oplus \mathscr{E}^-$ is graded as an algebra and, since the trace form $\langle\,,\, \rangle$ is $\msQ$-graded, the $\End(V)_\beta\otimes T(\mathscr{E})$ component $\tilde{\mrL}^\pm_\beta$ of $\tilde{\mrL}^\pm:=(1\otimes \imath^\mp)(\Omega^\pm)$ satisfies 
\begin{equation*}
\tilde{\mrL}^\pm_\beta\in \End(V)_\beta\otimes T(\mathscr{E})_\beta \quad \forall \quad \beta\in \msQ.
\end{equation*}
In particular, if we let $\End(V)^{t}$ denote $\End(V)$ equipped with its opposite $\msQ$-grading, in which $\End(V)^t_\beta:=t(\End(V)_\beta)=\End(V)_{-\beta}$ for all $\beta\in \msQ$, then $\tilde{\mrL}^\pm$ is a degree zero element of the $\msQ$-graded algebra $\End(V)^t\otimes T(\mathscr{E})$.
Next, let $(\eps_1,\eps_2)$ take value $(\pm,\pm)$ or $(+,-)$. Since, the elements
\begin{equation*}
\tilde{\Lz}^\pm\tilde{\Lz}^\mp-I \quad \text{ and }\quad \mrR\tilde{\mrL}_1^{\eps_1}\tilde{\mrL}_2^{\eps_2}-\tilde{\mrL}_2^{\eps_2}\tilde {\mrL}_1^{\eps_1}\mrR
\end{equation*}
are homogeneous of degree zero in the $\msQ$-graded algebras $\End(V)^t\otimes T(\mathscr{E})$ and $ (\End(V)^t)^{\otimes 2}\otimes T(\mathscr{E})$, the ideal in $T(\mathscr{E})$ generated by the relations \eqref{Lz+Lz-}--\eqref{RLL} is graded. Therefore, we may conclude that $\URg$ is a $\msQ$-graded algebra with $\mrL_\beta^\pm$  laying in $\End(V)_\beta\otimes \URg_\beta$ for each $\beta \in \msQ_\pm$.  
Moreover, since the matrices 
\begin{equation*}
\mrL^\pm, \quad (\mrL^\pm)^{-1}, \quad \mrL_{[1]}^\pm \mrL_{[2]}^\pm\quad \text{ and }\quad I 
\end{equation*}
are homogeneous elements of degree zero in the $\msQ$-graded algebras $\End(V)^t\otimes \URg$,  $\End(V)^t\otimes \URg$, $\End(V)^t\otimes \URg^{\otimes 2}$ and $\End(V)^t$, respectively, it follows from \eqref{URg-Hopf} that the structure maps $\Delta_\mrR$, $S_\mrR$ and $\veps_\mrR$ are $\msQ$-graded. \qedhere
\end{proof}
To conclude this subsection, we establish some basic properties satisfied by the weight zero components $\Lz^+$ and $\Lz^-$ of $\mrL^+$ and $\mrL^-$. 
\begin{lemma}\label{L:Lz-basic}
The matrix $\Lz:=\Lz^+$ satisfies the relations
\begin{equation}
\label{Lz:basic}
\Lz_1\Lz_2=\Lz_2\Lz_1,\quad \mrR\Lz_1\Lz_2=\Lz_2\Lz_1\mrR\quad \text{ and }\quad \Lz_1\mrL_2^\pm\Lz_1^{-1}=\Rd^{-1}\mrL_2^\pm \Rd.
\end{equation}
In particular, the coefficients of $\Lz^{\pm 1}$ generate a commutative subalgebra of $\URg$.
\end{lemma}
\begin{proof}
The second relation follows by projecting $\mrR \mrL_1^+\mrL_2^+=\mrL_2^+\mrL_1^+\mrR$ onto the summand  $\End(V)^{\otimes 2} \otimes \URg_0$, while the third relation follows by projecting that same identity onto $\End(V)_0\otimes \End(V)\otimes \URg$. 
As $\Lz_2$ commutes with $\Rd$, the identity $\Lz_1\Lz_2=\Lz_2\Lz_1$ follows by taking the $\URg_0$ component of the third relation of \eqref{Lz:basic}. \qedhere
\end{proof}

\subsection{Automorphisms}\label{ssec:URg-aut}

We now turn towards describing a natural family of algebra automorphisms of $\URg$. Let $\mathrm{G}_\mrR(V)$ denote the group
\begin{equation*}
\mathrm{G}_\mrR(V)=\{D\in \GLgr:\mathrm{Ad}(D\otimes D)(\mrR)=\mrR\},
\end{equation*}
where we recall from Section \ref{ssec:Uqg-reps} that $\GLgr=\prod_{\mu}\mathrm{GL}(V_\mu)$ is the group of invertible, degree zero linear automorphisms of the vector space $V$. In addition, we recall that $A\in \GLgr$ is the symmetric matrix defined in Lemma \ref{L:A-def}.
\begin{proposition}\label{P:R-aut}
Let $D\in \mathrm{G}_\mrR(V)$ and $\zeta\in \Q(q)^\times$. Then there exists $\Q(q)$-algebra automorphisms  $\upvartheta_D, \upchi_\zeta$ and $\chev$ of $\URg$	uniquely determined by  
\begin{equation*}
\upvartheta_D(\mrL^\pm)=D\mrL^\pm D^{-1}, \quad \upchi_\zeta(\mrL^\pm)=\mrL^\pm \zeta^{\pm 1}
\quad \text{ and }\quad \chev(\mrL^\pm)=A^{-1}(\mrL^\mp)^t A.
\end{equation*}
 Moreover, $\upvartheta_D$ is a Hopf algebra automorphism satisfying $\upvartheta_D^{-1}=\upvartheta_{D^{-1}}$, while $\chev$ is a coalgebra anti-automorphism intertwining $S_\mrR$ and $S^{-1}_\mrR$ and satisfying $\chev^{-1}=\chev$.
\end{proposition}
\begin{proof}
We will prove that the assignment $\chev(\mrL^\pm)=A^{-1}(\mrL^\mp)^t A$ extends to an automorphism of $\URg$ with the claimed properties; we refer the reader to the proof of Proposition 5.4 in \cite{GRW-RQUE} for the assertions involving $\upvartheta_D$ and $\upchi_\zeta$. 

First note that, since $t(\End(V)_\beta)=\End(V)_{-\beta}$ (see \eqref{t-reverse}) and $A\in \GLgr$, we have 
\begin{equation*}
\chev(\mrL^\pm_\beta)=A^{-1}(\mrL^\mp_{-\beta})^t A \quad \forall\quad \beta\in \msQ.
\end{equation*}
In particular, $\chev(\Lz^\pm)=A^{-1}(\Lz^\pm)^t A$. We thus have 
\begin{equation*}
\chev(\Lz^\pm)\chev(\Lz^\mp)=A^{-1}(\Lz^\pm)^t(\Lz^\mp)^tA=I,
\end{equation*}
where in the last equality we have used that, 
 by Lemma \ref{L:Lz-basic}, the coefficients of $\Lz^+$ and $\Lz^-$ generate a commutative subalgebra of $\URg$, and therefore $(\Lz^\pm)^t(\Lz^\mp)^t=(\Lz^\mp \Lz^\pm)^t=I$. Thus, the assignment $\chev$ preserves the relation \eqref{Lz+Lz-}. 
 
To show that the relations \eqref{RLL+-} and \eqref{RLL} are preserved, suppose more generally that $\msS,\msT\in \End(V)\otimes \URg$ satisfy $\mrR \msS_1 \msT_2=\msT_2 \msS_1 \mrR$. Let us show $\chev(S):=A^{-1}\msT^t A$ and $\chev(T):=A^{-1}\msS^t A$ do as well. Using that $\mrR^{t_1,t_2}_{21}=A_1A_2 \mrR A_1^{-1}A_2^{-1}$ (see Lemma \ref{L:R^t}), we obtain
\begin{align*}
\mrR \chev(\msS)_1 \chev(\msT)_2
&=
\mrR A_1^{-1} \msT_1^t A_1 A_2^{-1} \msS_2^t A_2 
\\
&=\mrR A_1^{-1}A_2^{-1}  \msT_1^t \msS_2^t A_1 A_2\\
&
=A_1^{-1}A_2^{-1}\mrR_{21}^{t_1,t_2} \msT_1^t \msS_2^t A_1 A_2\\
&
= A_1^{-1}A_2^{-1} (\msT_2 \msS_1 \mrR)^{t_1,t_2}_{21}A_1 A_2\\
&
= \chev(\msT)_2 \chev(\msS)_1  A_2^{-1}A_1^{-1}\mrR_{21}^{t_1,t_2}A_1 A_2
=
\chev(\msT)_2 \chev(\msS)_1 \mrR,
\end{align*}
as desired. Hence, we may conclude that the assignment $\chev(\mrL^\pm)=A^{-1}(\mrL^\mp)^tA$ uniquely extends to an algebra endomorphism of $\URg$. Moreover, since $A$ is symmetric, we have 
\begin{equation*}
\chev^2(\mrL^\pm)=\chev(A^{-1}(\mrL^\mp)^tA)= A^{-1} A^t \mrL^\pm  (A^{-1})^t A=\mrL^\pm,
\end{equation*}
which shows that $\chev$ is an involution. 

We are left to establish that $\chev$ is a coalgebra anti-automorphism satisfying $\mathrm{S}_\mrR^{-1}\circ \chev =\chev \circ \mathrm{S}_\mrR$. This is equivalent to the assertion that $\chev$ is a Hopf algebra isomorphism $\URg\iso \URg^{\mathrm{cop}}$, where $\URg^{\mathrm{cop}}$ is the co-opposite Hopf algebra to $\URg$. By the uniqueness of the antipode, it is sufficient to verify that 
\begin{equation*}
\Delta^{\mathrm{op}}_\mrR \circ \chev = (\chev \otimes \chev) \circ \Delta_\mrR \quad \text{ and }\quad \veps_\mrR\circ \chev=\veps_\mrR.
\end{equation*}
The latter identity clearly holds; let's establish the former. We have 
\begin{equation*}
\Delta^{\mathrm{op}}_\mrR (\chev(\mrL^\pm))= \Delta^{\mathrm{op}}_\mrR (A^{-1}(\mrL^\mp)^t A)
=
A^{-1} (\mrL^\mp_{[2]}\mrL^\mp_{[1]})^t A=A^{-1} (\mrL^\mp_{[1]})^tA \cdot A^{-1}(\mrL^\mp_{[2]})^t A,
\end{equation*}
which is precisely $(\chev \otimes \chev)(\Delta_\mrR(\mrL^\pm))$, as desired. \qedhere
\end{proof}
\begin{remark}
We shall henceforth refer to $\chev$ as the \textit{Chevalley involution} on $\URg$. It has the property that 
\begin{equation*}
\chev(\URg_\beta)\subset \URg_{-\beta} \quad \forall \quad \beta \in \msQ
\end{equation*}
and, as a consequence of Theorem \ref{T:Gen} and the results of Section \ref{sec:URg=Uq}, it indeed corresponds to an involution on $\bUqg$ extending the standard Chevalley involution $\omega$ on $U_q(\mfg)$.
\end{remark}
We now consider a particular subclass of the automorphisms $\upvartheta_D$ related to the $\msQ$-grading introduced in Proposition \ref{P:grading}.
Observe that there is an inclusion of groups $\msQ\into \mathrm{G}_\mrR(V)$ given by
\begin{equation*}
\beta \mapsto \pi(K_\beta)=\sum_{\mu}q^{(\beta,\mu)}\mathrm{Id}_{V_\mu} \quad \forall\; \beta\in \msQ.
\end{equation*}
Thus, we get an action of $\msQ$ on $\URg$ by Hopf algebra automorphisms, given by $\beta\mapsto \upvartheta_\beta:=\upvartheta_{\pi(K_\beta)}$. Moreover, one has 
\begin{equation*}
\upvartheta_\alpha(\mrL_\beta^\pm)=q^{(\alpha,\beta)}\mrL_\beta^\pm  \quad \forall \quad \alpha,\beta\in \msQ.
\end{equation*}
It follows that the $\msQ$-grading on $\URg$ introduced in the previous subsection is precisely the simultaneous eigenspace decomposition of the family $\{\upvartheta_\beta\}_{\beta\in \msQ}$. That is, for each $\beta\in \msQ$, one has
\begin{equation*}
\URg_\beta=\{x\in \URg: \upvartheta_\alpha(x)=q^{(\alpha,\beta)}x \quad \forall\; \alpha\in \msQ\}.
\end{equation*}
We note that this is entirely analogous to the situation which unfolds for the $\hbar$-adic analogue of $\URg$ studied in \cite{GRW-RQUE}; see Corollary 5.9 therein.

\subsection{The semisimplicity of $\Lz^\pm$}\label{ssec:Lz-semi}
The goal of this subsection is to prove that the matrices $\Lz^+$ and $\Lz^-$ are diagonal with a single eigenvalue associated to each weight space $V_\mu$. To make this precise, note that since $\End(V)_0=\bigoplus_\mu \End(V_\mu)$ and $\Lz^-=(\Lz^+)^{-1}$, the matrices $\Lz^\pm$ admit decompositions 
\begin{equation*}
\Lz^\pm=\sum_{\mu} \Lz_\mu ^{\pm 1} \quad \text{ with }\quad \Lz_\mu\in \End(V_\mu)\otimes \URg,
\end{equation*}
where the summation is taken over all weights $\mu$ of $V$. 
\begin{theorem}\label{T:Lz}
For each weight $\mu$ of $V$, there exists a unique invertible element $\mathscr{l}_\mu$ of $\URg$ such that 
\begin{equation*}
\Lz_\mu=\mathrm{Id}_{V_\mu}\otimes \mathscr{l}_\mu.
\end{equation*}
Moreover, these elements are grouplike and satisfy
\begin{equation*}
\mathscr{l}_\mu x_\beta \mathscr{l}_\mu^{-1}=q^{-(\mu,\beta)}x_\beta \quad \forall \quad x_\beta\in \URg_\beta.
\end{equation*}
\end{theorem}
\begin{proof}
 
By projecting the third relation of \eqref{Lz:basic} from Lemma \ref{L:Lz-basic} onto $\End(V_\mu)\otimes \End(V)_\beta\otimes \URg$ for any $\beta\in \msQ$ and weight $\mu$ of $V$, we obtain 
\begin{equation*}
\Lz_{\mu,1}(\mrL_\beta^\pm)_2\Lz_{\mu,1}^{-1}=\Rd^{-1}|_{V_\mu\otimes V}(\mrL_\beta^\pm)_2 \Rd|_{V_\mu\otimes V}=q^{-(\mu,\beta)}(\mrL_\beta^\pm)_2,
\end{equation*}
where in the second equality we have used that $\Rd|_{V_\mu\otimes V_\gamma}=	q^{(\mu,\gamma)-(\lambda,\lambda)}\mathrm{Id}_{V_\mu\otimes V_\gamma}$. The above identity implies that, if $\Lz_\mu=\mathrm{Id}_{V_\mu}\otimes \mathscr{l}_\mu$, then one has $\mathscr{l}_\mu x_\beta \mathscr{l}_\mu^{-1}=q^{-(\mu,\beta)}x_\beta$ for all $x_\beta \in \URg_\beta$. Similarly, if the first part of the theorem holds then, since $\Delta_\mrR(\Lz)=\Lz_{[1]}\Lz_{[2]}$, each $\mathscr{l}_\mu$ will automatically be grouplike.

 Hence, we are left to establish the first assertion of the theorem, which requires a more careful treatment. To begin, fix a positive integer $n$ and,  
for each $n$-tuple $(\beta_j)_{j=1}^n \in (\msQ_+)^n$, set 
\begin{equation*}
\mrR_{\beta_1,\ldots,\beta_n}:=(\mrR_{\beta_1})_{1,n+1}\cdots (\mrR_{\beta_n})_{n,n+1} \in \left(\otimes_{j=1}^n \End(V)_{\beta_j}\right)\otimes \End(V)_{-\sum_j \beta_j}.
\end{equation*}
 Given an $n$-tuple of weights $(\mu_j)_{j=1}^n$ of $V$ and an auxiliary weight $\gamma$ of $V$, we set 
\begin{gather*}
\mrR_{\beta_1,\ldots,\beta_n}^{\mu_1,\ldots,\mu_n}(\gamma):=\mrR_{\beta_1,\ldots,\beta_n}|_{V_{\mu_1}\otimes \cdots \otimes V_{\mu_n}\otimes V_\gamma},\\
\Lz_{\mu_1,\ldots,\mu_n}:
= \Lz_{\mu_1,1}\cdots \Lz_{\mu_n,n}.
\end{gather*}

\noindent \textit{Claim}. For each $(\beta_j)_{j=1}^n$, $(\mu_j)_{j=1}^n$ and $\gamma$, as above, we have 
\begin{equation}\label{eq:special}
\begin{aligned}
\Lz_{\gamma-\sum_{j=1}^n \beta_j, n+1} \cdot & \mrR_{\beta_1,\ldots,\beta_n}^{\mu_1,\ldots,\mu_n}(\gamma)\\
&=
\Lz_{\mu_1+\beta_1,\ldots,\mu_n+\beta_n}^{-1} \cdot 
\mrR_{\beta_1,\ldots,\beta_n}^{\mu_1,\ldots,\mu_n}(\gamma)\cdot  \Lz_{\mu_1,\ldots,\mu_n}\cdot  \Lz_{\gamma, n+1}
\end{aligned}
\end{equation}
\begin{proof}[Proof of claim] \let\qed\relax
By applying the relation $\mrR\Lz_1\Lz_2=\Lz_2\Lz_1\mrR$ of Lemma \ref{L:Lz-basic} repeatedly (where $\Lz=\Lz^+$), we deduce that for any $n>0$ one has
\begin{equation*}
\mrR_{V^{\otimes n},V}\,  \Lz_1\Lz_2\cdots \Lz_n\cdot  \Lz_{n+1}=\Lz_{n+1}\cdot \Lz_1\Lz_2\cdots \Lz_n \,\mrR_{V^{\otimes n},V},
\end{equation*}
where $\mrR_{V^{\otimes n},V}:=\mrR_{1,n+1}\mrR_{2,n+1}\cdots \mrR_{n,n+1}$.  Left-multiplying this relation by  the inverse of $\Lz_{V^{\otimes n}}:=\Lz_1\Lz_2\cdots \Lz_n$ while using that the coefficients of $\Lz$ commute, we obtain
\begin{equation*}
\Lz_{V^{\otimes n}}^{-1}\mrR_{V^{\otimes n},V}\,  \Lz_{V^{\otimes n}}\cdot  \Lz_{n+1}=\Lz_{n+1}\cdot  \,\mrR_{V^{\otimes n},V}.
\end{equation*}
The formula \eqref{eq:special} is obtained from this relation by projecting onto the weight $(\beta_1,\ldots,\beta_n,-\sum_j \beta_j)$ component of $\End(V)^{\otimes (n+1)}$ and then restricting the resulting equality of operators to 
$V_{\mu_1}\otimes \cdots \otimes V_{\mu_n}\otimes V_\gamma$. \qedhere
\end{proof}

With the claim at our disposal, we now have the ingredients necessary to prove the first part of the theorem. First note that since $V_\lambda$ is one-dimensional, the desired result holds for $\Lz_\lambda$: it is of the form $\mathrm{Id}_{V_\lambda}\otimes \mathscr{l}_\lambda$ for a unique, degree zero, invertible element $\mathscr{l}_\lambda\in \URg$.

Let us now take $\mu$ to be any weight of $V$, and write $\lambda-\mu=\sum_{j=1}^n \alpha_{i_j}$ for some $n$ and simple roots $\alpha_{i_1},\ldots,\alpha_{i_n}$. Taking $\beta_j=\alpha_{i_j}$ and $\gamma=\lambda$ in formula \eqref{eq:special}, we obtain 
\begin{equation}\label{eq:special2}
\begin{aligned}
\Lz_{\mu, n+1} \cdot & \mrR_{\alpha_{i_1},\ldots,\alpha_{i_n}}^{\mu_1,\ldots,\mu_n}(\lambda)\\
&=
\Lz_{\mu_1+\alpha_{i_1},\ldots,\mu_n+\alpha_{i_n}}^{-1} \cdot 
\mrR_{\alpha_{i_1},\ldots,\alpha_{i_n}}^{\mu_1,\ldots,\mu_n}(\lambda)\cdot  \Lz_{\mu_1,\ldots,\mu_n}\cdot  \mathscr{l}_\lambda.
\end{aligned}
\end{equation}
Moreover, since $\Rp_{\alpha_i}=(q_i-q_i^{-1})\pi(E_i)\otimes \pi(F_i)$ and $\mrR_{\alpha_i}=\Rd \Rp_{\alpha_i}$, the operator $\mrR_{\alpha_{i_1},\ldots,\alpha_{i_n}}^{\mu_1,\ldots,\mu_n}(\lambda)$ is a nonzero scalar multiple of 
\begin{equation*}
\pi(E_{i_1})|_{V_{\mu_1}}\otimes \cdots \otimes \pi(E_{i_n})|_{V_{\mu_n}}\otimes \pi(F_{i_1}\cdots F_{i_n} )|_{V_{\gamma}}.
\end{equation*}
For each $1\leq j\leq n$, let us now choose $\mu_j$ with the property that $\pi(E_{i_j})|_{V_{\mu_j}}$ is nonzero. Such a $j$ exists since $\pi(E_{i_j})$ is not the zero operator on $V$. In addition, we choose $f_j\in \mathrm{Hom}(V_{\mu},V_{\mu+\alpha_{i_j}})^\ast$ such that $f_j(\pi(E_{i_j})|_{V_{\mu_j}})=1$. Applying $f_1\otimes \cdots \otimes f_n$ to the first $n$-tensor factors of \eqref{eq:special2} while evaluating the $(n+1)$-th factor on the highest weight vector $v_\lambda$  yields 
\begin{equation}\label{eq:eigen}
\Lz_{\mu} \cdot  \pi(F_{i_1}\cdots F_{i_n} )v_\lambda 
= \pi(F_{i_1}\cdots F_{i_n} )v_\lambda \otimes \mathscr{l}_{\mu_1,\ldots,\mu_n}^{\alpha_{i_1},\ldots,\alpha_{i_n}}  \mathscr{l}_\lambda
\end{equation}
where $\mathscr{l}_{\mu_1,\ldots,\mu_n}^{\alpha_{i_1},\ldots,\alpha_{i_n}}\in \URg$ is determined by 
\begin{equation}\label{eq:eigen2}
\begin{aligned}
&\mathscr{l}_{\mu_1,\ldots,\mu_n}^{\alpha_{i_1},\ldots,\alpha_{i_n}}\\
&=(f_1\otimes \cdots \otimes f_n\otimes \mathrm{Id}_{\URg})\left(\Lz_{\underline{\mu}+\underline{\alpha}}^{-1} \cdot 
\pi(E_{i_1})|_{V_{\mu_1}}\otimes \cdots \otimes \pi(E_{i_n})|_{V_{\mu_n}}\cdot  \Lz_{\underline{\mu}}\right),
\end{aligned}
\end{equation}
with $\Lz_{\underline{\mu}+\underline{\alpha}}=\Lz_{\mu_1+\alpha_{i_1},\ldots,\mu_n+\alpha_{i_n}}$ and $\Lz_{\underline{\mu}}=\Lz_{\mu_1,\ldots,\mu_n}$. 
In particular, \eqref{eq:eigen} implies that each vector $\pi(F_{i_1}\cdots F_{i_n} )v_\lambda$ is an eigenvector for the matrix $\Lz_\mu$ with coefficients in $\URg$, with eigenvalue $\mathscr{l}_{\mu_1,\ldots,\mu_n}^{\alpha_{i_1},\ldots,\alpha_{i_n}}  \mathscr{l}_\lambda$. Moreover, the above formula for $\mathscr{l}_{\mu_1,\ldots,\mu_n}^{\alpha_{i_1},\ldots,\alpha_{i_n}}$ implies that 
\begin{equation*}
\mathscr{l}_{\mu_{\sigma(1)},\ldots,\mu_{\sigma(n)}}^{\alpha_{i_{\sigma(1)}},\ldots,\alpha_{i_{\sigma(1)}}}  =\mathscr{l}_{\mu_1,\ldots,\mu_n}^{\alpha_{i_1},\ldots,\alpha_{i_n}}  
\end{equation*}
for every permutation $\sigma\in \mathfrak{S}_n$. 
Since $V_{\mu}$ is spanned by all vectors of the form $\pi(F_{i_1}\cdots F_{i_n} )v_\lambda$ with $\sum_{j=1}^n \alpha_{i_j}=\lambda-\mu$, we see that every vector in $V_\mu$ is an eigenvector for $\Lz_\mu$ with the same eigenvalue, which is  
\begin{equation}\label{l_mu}
\mathscr{l}_\mu:=\mathscr{l}_{\mu_1,\ldots,\mu_n}^{\alpha_{i_1},\ldots,\alpha_{i_n}}  \mathscr{l}_\lambda
\end{equation}
for any sequence of simple roots $\alpha_{i_1},\ldots,\alpha_{i_n}$ which sum to $\lambda-\mu$ and weights $\mu_j$ with $\pi(E_{i_j})|_{V_{\mu_j}}\neq 0$. This completes the proof of the theorem. \qedhere
\end{proof}
For each $i\in \mbI$, let $\wlattice_{\lambda,i}\subset \wlattice_\lambda$ denote the (non-empty) set of all weights $\mu$ of $V$ for which $\mu+\alpha_i$ is also a weight:
\begin{equation*}
\wlattice_{\lambda,i}=\{\mu\in \wlattice_\lambda: \mu+\alpha_i\in \wlattice_\lambda\}.
\end{equation*}
Let us now fix $\mu\in \wlattice_\lambda$ to be any weight of $V$. Returning to  \eqref{eq:eigen2} and \eqref{l_mu} with the above theorem at our disposal, we obtain
\begin{equation*}
\mathscr{l}_\mu:=\mathscr{l}_{\mu_1,\ldots,\mu_n}^{\alpha_{i_1},\ldots,\alpha_{i_n}}  \mathscr{l}_\lambda
=
\mathscr{l}_\lambda\prod_{j=1}^n \mathscr{l}_{\mu_j+\alpha_{i_j}}^{-1}\mathscr{l}_{\mu_j},
\end{equation*}
where $(\alpha_{i_j})_{j=1}^n$ are simple roots such that $\lambda-\mu=\sum_{j=1}^n \alpha_{i_j}$ and each $\mu_j$ is an arbitrary element of $\wlattice_{\lambda,i_j}$. In particular, if $j\in \mbI$ is fixed and $\mu,\gamma$ are any two weights of $V$ in $\wlattice_{\lambda,j}$, then
\begin{equation*}
\mathscr{l}_{\mu+\alpha_j}^{-1}\mathscr{l}_\mu=\mathscr{l}_{\gamma+\alpha_j}^{-1}\mathscr{l}_\gamma.
\end{equation*}
This last identity can also be deduced from Theorem \ref{T:Lz} by projecting the relation $\mrR\Lz_1\Lz_2=\Lz_2\Lz_1\mrR$ of Lemma \ref{L:Lz-basic} onto $\Hom(V_\mu,V_{\mu+\alpha_j})\otimes \Hom(V_{\gamma+\alpha_j},V_\gamma)\otimes \URg$  and using that the operators $\pi(E_j)|_{V_\mu}$ and $\pi(F_j)|_{V_{\gamma+\alpha_j}}$ are nonzero for $\mu,\gamma\in \wlattice_{\lambda,j}$. 
The next corollary summarizes these observations. 
\begin{corollary}\label{C:upxi}
For each $i\in \mbI$, there is a unique element $\upxi_i\in \URg$ with the property that
\begin{equation*}
\upxi_i=\mathscr{l}_{\mu+\alpha_i}^{-1}\mathscr{l}_\mu \quad \forall \quad \mu \in \wlattice_{\lambda,i}.
\end{equation*}
Moreover, if $\mu$ is any weight of $V$ and $\lambda-\mu=\beta$ with $\beta=\sum_{j\in \mbI}n_j \alpha_j$, then 
\begin{equation*}
\mathscr{l}_\mu=\mathscr{l}_\lambda 	\upxi_\beta, \quad \text{ where }\quad \upxi_\beta:=\prod_{j\in \mbI}\upxi_j^{n_j}.
\end{equation*}
\end{corollary}
To conclude this subsection, we note that, as an immediate consequence of the second assertion of Theorem \ref{T:Lz} and the above corollary, we have 
\begin{equation}\label{Ad-upxi}
\upxi_i x_\beta\upxi_i^{-1}= q^{(\alpha_i,\beta)}x_\beta \quad \forall \quad  x_\beta\in \URg_\beta \; \text{ and }\; \beta\in\msQ. 
\end{equation}
\subsection{The matrices \texorpdfstring{$\mathbb{L}^\pm$}{bbL\^\pm}}\label{ssec:URg-nil}
Let us now introduce the unipotent matrices $\mbbL^\pm$ by setting
\begin{equation*}
\mbbL^+:=\Lz^{-1}\mrL^+ \quad \text{ and }\quad \mbbL^-:=\mrL^- \Lz,
\end{equation*}
where we have set $\Lz=\Lz^+$, as in the previous section. The goal of this section is to prove Theorem \ref{T:Gen}, which establishes that the components $\mbbL^\pm_\beta$ of $\mbbL^\pm$ lay in $\pi(U_q(\mfn^\pm)_\beta)\otimes \URg$ and have coefficients belonging to the subalgebras generated by the coefficients of $\mbbL^+_{\alpha_i}$ and  $\mbbL^-_{- \alpha_i}$ for all $i\in \mbI$. To prove this theorem we will need two lemmas, the first of which spells out the basic commutation relations satisfied by $\mbbL^\pm$.
\begin{lemma}\label{L:mbbL}
The matrices $\mbbL^\pm$ satisfy the relations 
\begin{align*}
\Rp\! \left(\Rd \mbbL_1^\pm \Rd^{-1}\right) \mbbL_2^\pm&=\mbbL_2^\pm \left(\Rd^{-1} \mbbL_1^\pm \Rd\right)\Rp\\
 \Rp_{\Lz_1}\mbbL_1^+ \mbbL_2^-&=\mbbL_2^- \mbbL_1^+ \Rp_{\Lz_2} 
\end{align*}
where $\Rp_{\Lz_i}=\Lz_i^{-1} \Rp \Lz_i$ for $i\in \{1,2\}$. 
\end{lemma}
\begin{proof}  Both of these identities follow from the defining relations of $\URg$ and Lemma \ref{L:Lz-basic}. For a proof of the first relation, we refer the reader to Lemma \ref{L:LVn} below, which establishes a more general family of identities. To prove the second relation, we apply \eqref{RLL+-} to obtain
\begin{align*}
\Rp_{\Lz_1}\mbbL_1^+ \mbbL_2^-&=\Lz_1^{-1}\Rd^{-1} \mrR \mrL_1^+ \mrL_2^- \Lz_2
=\Lz_1^{-1}\Rd^{-1} \mrL_2^- \mrL_1^+\mrR\Lz_2 = \Lz_1^{-1}\Rd^{-1} \mrL_2^- \mrL_1^+\Rd\Lz_2\Rp_{\Lz_2}. 
\end{align*}
The desired result now follows from Lemma \ref{L:Lz-basic} which gives $\Rd^{-1} \mrL_2^- \Rd=\Lz_1 \mrL_2^- \Lz_1^{-1}$ and $\Rd^{-1} \mrL_1^+ \Rd=\Lz_2 \mrL_1^+ \Lz_2^{-1}$, and hence 
\begin{equation*}
\Lz_1^{-1}\Rd^{-1} \mrL_2^- \mrL_1^+\Rd\Lz_2\Rp_{\Lz_2} =\mrL_2^-\Lz_1^{-1}\Lz_2\mrL_1^+ \Rp_{\Lz_2}
=\mbbL_2^- \mbbL_1^+ \Rp_{\Lz_2}. \qedhere
\end{equation*}
\end{proof}
To state the second lemma, recall from Section \ref{ssec:g} that, for each $\beta \in \dot\msQ_+$, the set $\dot\msQ_+^\beta$ consists of all $\alpha\in \dot\msQ_+$ for which $\beta-\alpha\in \dot\msQ_+$. 
The below result establishes some key commutation relations satisfied by the components of $\mbbL^+$. 
\begin{lemma}\label{L:Rec}
For each $\beta \in \dot\msQ_+$, the matrix $\mathbb{L}_\beta^+ \in \End(V)_\beta\otimes \URg$ satisfies the following two identities:
\begin{gather}
[\Rp_{\alpha_i},(\mbbL_\beta^+)_2]
=(\mbbL_{\beta-\alpha_i}^+)_2 (\mbbL_{\alpha_i}^+)_1\pi(K_i^{-1})_2 - (\mbbL_{\alpha_i}^+)_1 \pi(K_i)_2 (\mbbL_{\beta-\alpha_i}^+)_2,\label{Rec:1}\\
(\mbbL_\beta^+)_1 \pi(K_\beta^{-1}-K_\beta)_2
=
[\Rp_\beta,(\mbbL_\beta^+)_2]+\sum_{\alpha\in \dot\msQ_+^\beta}\mathbb{L}_{\beta,\alpha}^+,
\label{Rec:2}
\end{gather}
where $\mathbb{L}_{\beta,\alpha}^+$ is defined by
\begin{equation*}
\mathbb{L}_{\beta,\alpha}^+=\Rp_\alpha (\mathbb{L}_{\beta-\alpha}^+)_1 \pi(K_{\beta-\alpha})_2 (\mathbb{L}_\alpha^+)_2
-
(\mathbb{L}_\alpha^+)_2(\mathbb{L}_{\beta-\alpha}^+)_1 \pi(K_{\beta-\alpha}^{-1})_2\Rp_\alpha.
\end{equation*}
%
\end{lemma}
\begin{proof}
Taking the $\End(V)_{\alpha_i}\otimes \End(V)_{\beta-\alpha_i}$-component of the first relation of Lemma \ref{L:mbbL}, we obtain
\begin{equation*}
\Rp_{\alpha_i}(\mbbL_{\beta}^+)_2+\Rd(\mathbb{L}_{\alpha_i}^+)_1\Rd^{-1} (\mbbL_{\beta-\alpha_i}^+)_2
=
(\mbbL_{\beta-\alpha_i}^+)_2 \Rd^{-1}(\mbbL_{\alpha_i}^+)_1 \Rd+(\mbbL_{\beta}^+)_2 \Rp_{\alpha_i}.
\end{equation*}
To complete the proof of \eqref{Rec:1}, it remains to note that for any $\nu\in \msQ_+$, one has 
\begin{equation}\label{eq:D-K}
\Rd(\mathbb{L}_{\nu}^+)_1\Rd^{-1}=(\mathbb{L}_{\nu}^+)_1\pi(K_\nu)_2.
\end{equation}

Let's now turn to the relation \eqref{Rec:2}. Taking instead the $\End(V)_\beta\otimes \End(V)_0$-component of the first relation of Lemma \ref{L:mbbL} yields
\begin{align*}
\Rd (\mathbb{L}_{\beta}^+)_1\Rd^{-1} + &\Rp_\beta(\mbbL_\beta^+)_2 +\sum_{\alpha\in \dot\msQ_+^\beta}\Rp_\alpha\Rd(\mathbb{L}_{\beta-\alpha}^+)_1\Rd^{-1} (\mbbL_\alpha^+)_2\\
&=
\Rd^{-1}(\mbbL_{\beta}^+)_1 \Rd
+
(\mbbL_\beta^+)_2  \Rp_\beta
+
\sum_{\alpha\in \dot\msQ_+^\beta} (\mbbL_\alpha^+)_2 \Rd^{-1}(\mbbL_{\beta-\alpha}^+)_1 \Rd \Rp_\alpha.
\end{align*}
The relation \eqref{Rec:2} follows readily from this identity by rearranging and making use of \eqref{eq:D-K}. \qedhere
\end{proof}

We are now prepared to state and prove the main result of this subsection. Recall that $\chev$ is the Chevalley involution of $\URg$ defined in Proposition \ref{P:R-aut}.
\begin{theorem}\label{T:Gen}For each  $i\in \mbI$ there is an element $\Xr_i\in \URg_{\alpha_i}$ such that 
\begin{equation*}
\mathbb{L}_{\alpha_i}^+=\pi(E_i)\otimes \Xr_i \quad \text{ and }\quad \mathbb{L}_{-\alpha_i}^-=\pi(F_i)\otimes \Yr_i, 
\end{equation*}
where $\Yr_i=\chev(\Xr_i)\in \URg_{-\alpha_i}$. 
Moreover,  one has 
\begin{equation*}
\mathbb{L}_\beta^\pm\in \pi(U_q(\mfn^\pm)_\beta)\otimes \mbU_\mrR^\lambda(\mfn^\pm)_\beta \quad \forall \quad \beta\in \msQ_\pm,
\end{equation*}
where $\mbU_\mrR^\lambda(\mfn^+)$ and $\mbU_\mrR^\lambda(\mfn^-)$  denote the subalgebras of $\URg$ generated by $\{\Xr_i\}_{i\in \mbI}$ and $\{\Yr_i\}_{i\in \mbI}$, respectively.
\end{theorem}
\begin{proof} 
We will first show by induction on the height $\mathrm{ht}(\beta)$ of $\beta\in \msQ_+$ that one has 
\begin{equation}\label{Gen-step1}
\mbbL_\beta^+\in \pi(U_q(\mfn^+)_\beta)\otimes \URg_\beta \quad \forall\; \beta\in \msQ_+.
\end{equation}
 If $\beta=0$, then this is trivial as $\mathbb{L}_0=I$. Suppose now that $\beta\neq 0$ has height $k>0$. Then, we can find a weight $\mu$ of $V$ such that $(\beta,\mu)\neq 0$; indeed, otherwise $\beta$ annihilates the lattice $\wlattice(\lambda)\supset \msQ$, which is impossible. Let $v\in V_\mu$ be nonzero and let $f\in V^\ast$ be such that $f(v)=1$. Then, applying the linear functional $\varphi \in\End(V)^\ast$ defined by $\varphi(x)=f(xv)$ to the second tensor factor of \eqref{Rec:2} yields
\begin{equation}\label{Rec:3}
\mathbb{L}_\beta^+(q^{-(\beta,\mu)}-q^{(\beta,\mu)})= \varphi_2\left( [\Rp_\beta,(\mathbb{L}_\beta^+)_2]\right)+\sum_{\alpha\in \dot\msQ_+^\beta} \varphi_2(\mathbb{L}_{\beta,\alpha}^+),
\end{equation}
where $\varphi_2=\mathrm{Id}_V\otimes \varphi\otimes \mathrm{Id}_{\URg}$. 
Since $\alpha\in \dot\msQ_+^\beta$ implies that $\mathrm{ht}(\beta-\alpha)<\mathrm{ht}(\beta)$ and $\Rp_\nu\in \pi(U_q(\mfn^+))_\nu\otimes \pi(U_q(\mfn^-))_{-\nu}$ for each $\nu \in \msQ_+$, we can conclude from the above formula, the definition of $\mbbL_{\beta,\alpha}^+$ (see Lemma \ref{L:Rec}), and the inductive hypothesis that $\mathbb{L}_\beta^+$ belongs to $\pi(U_q(\mfn^+)_\beta)\otimes \URg_\beta$. This completes the proof of \eqref{Gen-step1}.

Since $\pi(U_q(\mfn^+)_{\alpha_i})=\Q(q)\cdot \pi(E_i)$ for each $i\in \mbI$, it follows immediately from \eqref{Gen-step1} that $\mathbb{L}_{\alpha_i}^+$ is of the form $\mathbb{L}_{\alpha_i}^+=\pi(E_i)\otimes \Xr_i$ for some $\Xr_i\in \URg_{\alpha_i}$. 

Next, we show that for arbitrary $\beta\in \dot\msQ_+$, one in fact has 
\begin{equation}\label{Gen-step2}
\mathbb{L}_\beta^+\in \pi(U_q(\mfn^+)_\beta)\otimes \mbU_\mrR^\lambda(\mfn^+)_\beta,
\end{equation}
where $\mbU_\mrR^\lambda(\mfn^+)$ is as in the statement of the theorem. For this, we can again induct on the height of $\beta$, with the base case being trivial. Assume now the assertion holds for any $\alpha\in \msQ_+$ of height less than $k$, and let $\beta \in\msQ_+$ be such that $\mathrm{ht}(\beta)=k$. Using \eqref{Rec:3} and the inductive hypothesis, we see that the coefficients of $\mathbb{L}_\beta^+$ will belong to $\mathrm{U}_\mrR^\lambda(\mfn^+)_\beta$ provided 
\begin{equation*}
[\Rp_\beta,(\mathbb{L}_\beta^+)_2]\in \End(V)^{\otimes 2}\otimes \mathbf{U}_\mrR^\lambda(\mfn^+)_\beta.
\end{equation*}
Since $\Rp_\beta\in \End(V)\otimes \pi(U_q(\mfn^-))_\beta$, this will be true provided 
\begin{equation*}
[\pi(F_{i_1}\cdots F_{i_{k}}),\mathbb{L}_\beta^+]\in \End(V)\otimes \mathbf{U}_\mrR^\lambda(\mfn^+)_\beta 
\end{equation*}
for all $(i_j)_{j=1}^k\in \mbI^k$ such that $\beta=\sum_{j=1}^k \alpha_{i_j}$. Since $\Rp_{\alpha_i}=(q_i-q_i^{-1})\pi(E_i)\otimes \pi(F_i)$, this is now a consequence of \eqref{Rec:1} and the inductive hypothesis.

To complete the proof of the theorem, we must show that $\mathbb{L}_{-\alpha_i}^-=\pi(F_i)\otimes \chev(\Xr_i)$ and that the counterpart of \eqref{Gen-step2} holds with $\mbbL^+$ replaced by $\mbbL^-$, $\msQ_+$ by $\msQ_-$, and $\mfn^+$ by $\mfn^-$. This can be accomplished using properties of the Chevalley involution $\chev$, as we now explain. First, observe that
\begin{equation*}
\chev(\mbbL^-)=\chev(\mrL^- \Lz)=A^{-1}(\mrL^+)^t A \Lz^{-1},
\end{equation*}
where we have used that $\chev(\Lz)=A^{-1}(\Lz^{-1})^t A=\Lz^{-1}$, which follows from the fact that, by Theorem \ref{T:Lz}, $\Lz$ is diagonal and commutes with $A$. Next, using properties of the transpose $t$ and Theorem \ref{T:Lz}, we deduce that
\begin{equation*}
(\mrL^+)^t_{-\beta} = (\Lz \mbbL^+_\beta)^t=(\mbbL_\beta)^t \pi(K_\beta^{-1})\Lz \quad \forall\quad \beta\in \msQ_+.
\end{equation*}
Combining the above formulae, we obtain 
\begin{equation}\label{chev:L}
\chev(\mbbL^-_{-\beta})=A^{-1}(\mbbL_\beta^+)^t  A  \cdot \pi(K_\beta^{-1}) \quad \forall \quad \beta\in \msQ_+.
\end{equation}
In the special case where $\beta=\alpha_i$, we may use that $\mbbL_{\alpha_i}^+=\pi(E_i)\otimes \Xr_i$ and that $A^{-1}\pi(x)^tA=\pi(\uptau(x))$ (see Lemma \ref{L:A-def}) to obtain 
\begin{equation*}
\chev(\mbbL^-_{-\alpha_i})=\pi(\uptau(E_i) K_i^{-1})\otimes \Xr_i=\pi(F_i)\otimes \Xr_i.
\end{equation*}
Since $\chev$ is involutive, this yields $\mbbL^-_{-\alpha_i}=\pi(F_i)\otimes \Yr_i$ with $\Yr_i=\chev(\Xr_i)$, as desired. 

More generally, if $\beta\in \msQ_+$ is arbitrary then from \eqref{Gen-step2} and \eqref{chev:L}  we obtain
\begin{equation*}
\chev(\mbbL^-_{-\beta})\in \pi(\uptau(U_q(\mfn^+)_\beta)K_\beta^{-1})\otimes \mbU_\mrR^\lambda(\mfn^+)_\beta =\pi(U_q(\mfn^-)_{-\beta})\otimes \mbU_\mrR^\lambda(\mfn^+)_\beta,
\end{equation*}
where we have used that $\uptau(U_q(\mfn^+)_\beta)K_\beta^{-1}=U_q(\mfn^-)_{-\beta}$; see \eqref{uptau}. Since $\chev$ is an involution and $\chev(\mbU_\mrR^\lambda(\mfn^+)_\beta)=\mbU_\mrR^\lambda(\mfn^-)_{-\beta}$, we may conclude that 
\begin{equation*}
\mbbL^-_{-\beta}\in \pi(U_q(\mfn^-)_{-\beta})\otimes \mbU_\mrR^\lambda(\mfn^-)_{-\beta} \quad \forall\quad \beta\in \msQ_+. \qedhere
\end{equation*}
\end{proof}
\begin{remark}\label{R:Gen}
As an application of Corollary \ref{C:upxi} and Theorems \ref{T:Lz} and \ref{T:Gen}, we may conclude that $\URg$ is generated as a $\Q(q)$-algebra by $\mathscr{l}_\lambda$, $\mathscr{l}_\lambda^{-1}$, and the family of elements $\{\upxi_i^{\pm 1}, \Xr_i,\Yr_i\}_{i\in \mbI}$. We will prove in Section \ref{sec:URg=Uq} that these elements satisfy the defining relations of $\bUqg$. 
\end{remark}
%
%

\section{Identifying  \texorpdfstring{$\mathbf{U}_\mathrm{R}^\lambda(\mfg)$}{U\_R\^{}lambda(g)} and \texorpdfstring{$\mathbf{U}_q^\lambda(\mfg)$}{U\_q\^{}lambda(g)}}\label{sec:URg=Uq}

In this section, we state and prove the main theorem of this article, which identifies the $\msQ$-graded Hopf algebras $\URg$ and $\bUqg$; see Theorem \ref{T:main}. Afterwards, we use this theorem to recover natural characterizations of the quantum algebras $U_q(\mfg)$, $U_q^\lambda(\mfg)$ and $U_q(\mfg)\otimes \Q(q)[v_\lambda^{\pm 1}]$ within the framework of the $R$-matrix formalism; see Corollaries \ref{C:Recover} and \ref{C:nlambda}.
\subsection{The isomorphism \texorpdfstring{$\URg\cong \bUqg$}{U\_Rg->U\_qg}}\label{ssec:URg=Uq}
Recall from Section \ref{ssec:bUqg-gen} that $\scL^+$ and $\scL^-$ are the generating matrices for $\bUqg$ defined by 
\begin{equation*}
\scL^+=v\scK^{-1} \cdot (\pi\otimes \omega)(\mathscr{R}^+) \quad \text{ and }\quad 
\scL^-=(\pi\otimes \omega)((\mathscr{R}^+_{21})^{-1})\cdot \scK v^{-1}.
\end{equation*}
 In addition, we recall from Proposition \ref{P:gen-rels} that $\upxi_\lambda^\pm$ denotes the element $K_\lambda^{\pm 1}v^{\mp 1}$ of $\bUqg$. The following theorem provides the main result of this section.
\begin{theorem}\label{T:main}
The assignment $\mrL^\pm\mapsto\scL^\pm$ uniquely extends to an isomorphism of $\msQ$-graded Hopf algebras 
\begin{equation*}
\Upupsilon: \URg\iso \bUqg.
\end{equation*}
The inverse $\Upupsilon^{-1}$ of $\Upupsilon$ is given by the following formulae for $i\in \mbI:$
\begin{gather*}
\Upupsilon^{-1}(K_i^{\pm 1})=\upxi_i^{\pm 1}, \quad \Upupsilon^{-1}(\upxi_\lambda^{\pm})=\mathscr{l}_\lambda^{\mp 1},\\
\Upupsilon^{-1}(E_i)=\frac{\Xr_i}{q_i^{-1}-q_i} \qquad \text{ and }\qquad \Upupsilon^{-1}(F_i)=\frac{\Yr_i}{q_i-q_i^{-1}}.
\end{gather*}
\end{theorem}
\begin{proof}
It follows directly from Proposition \ref{P:bUqg-L} that the assignment $\mrL^\pm\mapsto\scL^\pm$ extends to a Hopf algebra homomorphism $\Upupsilon: \URg\to \bUqg$. Moreover, since $\scL^\pm$ are degree zero elements of the $\msQ$-graded algebra $\End(V)^t\otimes \bUqg$ (see the proof of Proposition \ref{P:grading}), $\Upupsilon$ respects the underlying $\msQ$-gradings. 

Hence, we are left to verify that $\Upupsilon$ is invertible with inverse as specified in the statement of the theorem. We will do this by explicitly showing the stated formulas for $\Upupsilon^{-1}$ determine a $\Q(q)$-algebra homomorphism $\bUqg\to \URg$, and then explain why it is necessarily $\Upupsilon^{-1}$.

\noindent\textit{Claim}. The assignment
\begin{gather*}
K_i^{\pm 1}\mapsto \upxi_i^{\pm 1}, \quad \upxi_\lambda^{\pm}\mapsto\mathscr{l}_\lambda^{\mp 1},\quad E_i\mapsto \frac{\Xr_i}{q_i^{-1}-q_i} \quad \text{ and }\quad F_i\mapsto \frac{\Yr_i}{q_i-q_i^{-1}}
\quad \forall \; i\in \mbI
\end{gather*}
extends to an algebra homomorphism $\Upupsilon^\sharp:\bUqg\to \URg$. 

\begin{proof}[Proof of claim]
By Proposition \ref{P:gen-rels}, this amounts to showing that the elements $\mathscr{l}_\lambda^{\mp 1}$, $\upxi_i^{\pm 1}$, $(q_i^{-1}-q_i)^{-1}\Xr_i$ and $(q_i-q_i^{-1})^{-1}F_i$ satisfy the relations \eqref{dot:1}--\eqref{dot:4} for $\upxi_\lambda^\pm$, $\upxi_i^\pm$, $x_i^+$ and $x_i^-$, respectively.

The first identity of \eqref{dot:1} is trivially satisfied, while the second holds since $\{\upxi_i^{\pm 1}\}_{i\in \mbI}$ and $\mathscr{l}_\lambda^{\pm 1}$ belong to the commutative subalgebra of $\URg$ generated by the coefficients of $\Lz$ and $\Lz^{-1}$; see Lemma \ref{L:Lz-basic} and Corollary \ref{C:upxi}.  
The commutation relations of \eqref{dot:2} are a consequence of Theorem \ref{T:Lz}, the identity \eqref{Ad-upxi}, and that, by Theorem \ref{T:Gen}, the elements $\Xr_j$ and $\Yr_j$ have degree $\alpha_j$ and $-\alpha_j$, respectively. 

The relations \eqref{dot:3} and \eqref{dot:4} for $\{\Xr_i,\Yr_i\}_{i\in \mbI}$ are more subtle. They are equivalent to the relations
\begin{gather}
[\Xr_i,\Yr_j]=\delta_{ij}(q_i^{-1}-q_i)\left(\upxi_i-\upxi_i^{-1}\right),
\label{main-suff:4}\\
\sum_{b=0}^{1-a_{ij}}(-1)^b \sbinom{1-a_{ij}}{b}_{q_i} \Xr_i^b \Xr_j \Xr_i^{1-a_{ij}-b}=0,
\label{main-suff:5}\\
\sum_{b=0}^{1-a_{ij}}(-1)^b \sbinom{1-a_{ij}}{b}_{q_i} \Yr_i^b \Yr_j \Yr_i^{1-a_{ij}-b}=0,
\label{main-suff:6}
\end{gather}
where $q_i=q^{d_i}$ and $i\neq j$ in the last two relations. 
To prove \eqref{main-suff:4} recall that, by Lemma \ref{L:mbbL}, we have 
$
\Rp_{\Lz_1}\mbbL_1^+ \mbbL_2^-=\mbbL_2^- \mbbL_1^+ \Rp_{\Lz_2} 
$
where $\Rp_{\Lz_a}=\Lz_a^{-1} \Rp \Lz_a$ for $a\in \{1,2\}$. Taking the $\End(V)_{\alpha_i}\otimes \End(V)_{-\alpha_j}\otimes \URg$ component of this relation and rearranging, we obtain 
\begin{equation}\label{[L+,L-]}
\left[(\mbbL^+_{\alpha_i})_1,(\mbbL^-_{-\alpha_j})_2\right]
=
\delta_{ij}\left((\Rp_{\Lz_2})_{\alpha_i}- (\Rp_{\Lz_1})_{\alpha_i}\right).
\end{equation}
To see that this is equivalent to \eqref{main-suff:4}, note that, for each $\beta\in \msQ_+$ and $x\in \End(V)_{\pm\beta}$, one has 
\begin{equation*}
\Lz^{-1} x \Lz = \sum_\mu (\mathrm{Id}_{V_{\mu\pm\beta}}\circ  x \circ \mathrm{Id}_{V_{\mu}}) \otimes \mathscr{l}_{\mu\pm\beta}^{-1} \mathscr{l}_\mu =\sum_\mu (\mathrm{Id}_{V_{\mu\pm\beta}}\circ  x \circ \mathrm{Id}_{V_{\mu}}) \otimes \upxi_\beta^{\pm 1} =x\otimes \upxi_\beta^{\pm 1}
\end{equation*}
where $\upxi_\beta$ is as in Corollary \ref{C:upxi} (in particular, $\upxi_{\alpha_i}=\upxi_i$). 
Since $\Rp_{\alpha_i}=(q_i-q_i^{-1})\pi(E_i)\otimes \pi(F_i)$, this implies that \eqref{[L+,L-]} is equivalent to 
\begin{equation*}
\pi(E_i)\otimes \pi(F_i)\otimes [\Xr_i,\Yr_j]=\pi(E_i)\otimes \pi(F_i)\otimes \delta_{ij}(q_i^{-1}-q_i)\left(\upxi_i-\upxi_i^{-1} \right),
\end{equation*}
where we have used that, by Theorem \ref{T:Gen}, we have $\mbbL^+_{\alpha_i}=\pi(E_i)\otimes \Xr_i$ and $\mbbL^-_{-\alpha_j}=\pi(F_j)\otimes \Yr_j$. The relation \eqref{main-suff:4} now follows immediately. 

We are left to verify that the $q$-Serre relations \eqref{main-suff:5} and \eqref{main-suff:6} are satisfied. By Theorem \ref{T:Gen}, $\Yr_k=\chev(\Xr_k)$ for all $k\in \mbI$, so it is sufficient to establish that \eqref{main-suff:5} holds for all $i,j\in \mbI$ with $i\neq j$. This is proven
 in detail in Section \ref{ssec:Serre} below. \let\qed\relax
\end{proof}
Now let us explain why $\Upupsilon^\sharp$ is the inverse of $\Upupsilon$. 
By Remark \ref{R:Gen}, the elements $\mathscr{l}_\lambda^{\pm 1}$, $\upxi_i^{\pm 1}$, $\Xr_i$ and $\Yr_i$ ($i\in \mbI$) generate $\URg$. Hence, it is sufficient to show that $\Upsilon\circ \Upupsilon^\sharp=\mathrm{Id}_{\bUqg}$. This is a consequence of the following claim. 

\noindent\textit{Claim}. For each $i\in \mbI$, one has 
\begin{equation}\label{Ups-KM}
\begin{gathered}
\Upupsilon(\mathscr{l}_\lambda)=\upxi_\lambda^-=K_\lambda^{-1}v,\\
\Upupsilon(\upxi_i)=K_i,
\quad
\Upupsilon(\Xr_i)=(q_i^{-1}-q_i)E_i,\quad \Upupsilon(\Yr_i)=(q_i-q_i^{-1})F_i,
\end{gathered}
\end{equation}
\begin{proof}[Proof of claim]
Since $\Upupsilon(\Lz)=\scL^+_0=v\scK^{-1}$ and $\Lz=\sum_\mu \mathrm{Id}_{V_\mu}\otimes \mathscr{l}_\mu$, we have $\Upupsilon(\mathscr{l}_\mu)=vK_\mu^{-1}$ for each weight $\mu$ of $V$. This immediately gives $\Upupsilon(\mathscr{l}_\lambda)=\upxi_\lambda^-$, while also implying 
\begin{equation*}
\Upupsilon(\upxi_i)=\Upupsilon(\mathscr{l}_{\mu+\alpha_i}^{-1}\mathscr{l}_\mu)
=
K_{\mu+\alpha_i}v^{-1} v K_\mu^{-1}=K_i \quad \forall\; i\in \mbI.
\end{equation*}
Next, by the relations at the beginning of the proof of Proposition \ref{P:bUqg-L}, we have 
\begin{gather*}
\Upupsilon(\mbbL^+_{\alpha_i})=\Upupsilon(\mathscr{L}^{-1})\scL_{\alpha_i}^+=(q_i^{-1}-q_i)\,\pi(E_i)\otimes E_i,\\
\Upupsilon(\mbbL^-_{\alpha_i})=\scL_{\alpha_i}^- \Upupsilon(\mathscr{L})=(q_i-q_i^{-1})\,\pi(F_i)\otimes F_i,
\end{gather*}
from which it follows immediately that $\Upupsilon(\Xr_i)=(q_i^{-1}-q_i)E_i$ and $\Upupsilon(\Yr_i)=(q_i-q_i^{-1})F_i$ for each $i\in \mbI$. This completes the proof of the claim.\qedhere
%
%
\end{proof}
\let\qed\relax
\end{proof}
%
%
\subsection{The $q$-Serre relations}\label{ssec:Serre}

In this section, we will prove that the elements $\{\Xr_i\}_{i\in \mbI}$ satisfy the $q$-Serre relations \eqref{main-suff:5}; see Corollary \ref{C:Serre}. This  is essentially a consequence of the first relation of Lemma \ref{L:Rec}. However, to make this precise we will need to enlarge our underlying representation $V$ by introducing the $U_q(\mfg)$-representation 
\begin{equation*}
\mathbb{V}:=\bigoplus_{n\in \N}V^{\otimes n}.
\end{equation*}
Let $\pi_\mbbV:U_q(\mfg)\to \End(\mbbV)$ denote the associated $\Q(q)$-algebra homomorphism. An important property, which we will exploit in \eqref{veps+} below, is that $\mbbV$ is a faithful $U_q(\mfg)$-module (equivalently, $\pi_\mbbV$ is injective). This property can be deduced from its well-known classical analogue (for the enveloping algebra $U(\mfg)$) via a standard specialization argument. Two different proofs of the assertion for $U(\mfg)$ may be found in \cite{Pass-14}*{Thm.~4.10} and \cite{GRWvrep}*{Thm.~A.1}.

To simplify notation, we shall henceforth write $\mrL$ for $\mrL^+$ and $\mbbL$ for $\mbbL^+$. In addition, we set 
\begin{gather*}
\mrL_{V^{\otimes n}}:=\mrL_1\mrL_2\cdots \mrL_n \in \End(V)^{\otimes n}\otimes \URg,\\
\mrR_{V^{\otimes n},V}:=\mrR_{1,n+1}\mrR_{2,n+1}\cdots \mrR_{n,n+1}\in \End(V^{\otimes n})\otimes \End(V),\\
\mrR_{V,V^{\otimes n}}:=\mrR_{1,n+1}\mrR_{1,n}\cdots \mrR_{1,2}\in \End(V)\otimes \End(V^{\otimes n}).
\end{gather*}
Note that $\mrR_{V^{\otimes n},V}$ has already appeared in the proof of Theorem \ref{T:Lz}. 
The above definitions imply that
\begin{equation}\label{RLL:n}
\begin{aligned}
\mrR_{V^{\otimes n},V} \mrL_{V^{\otimes n}}\mrL_V= \mrL_V\mrL_{V^{\otimes n}}\mrR_{V^{\otimes n},V},\\
\mrR_{V,V^{\otimes n}}\mrL_V \mrL_{V^{\otimes n}}=\mrL_V \mrL_{V^{\otimes n}}\mrR_{V,V^{\otimes n}},
\end{aligned}
\end{equation}
in $\End(V^{\otimes n})\otimes \End(V)\otimes \URg$ and $\End(V)\otimes \End(V^{\otimes n})\otimes\URg$, respectively. These two spaces are of course isomorphic, but for the sake of notation it is convenient to view them as distinct; in particular $\mrL_V=\mrL_{n+1}$ in the first equality and $\mrL_V=\mrL_1$ in the second. 

Let $\Lz_{V^{\otimes n}}$, $\Rd_{V^{\otimes n},V}$ and $\Rd_{V,V^{\otimes n}}$ be the weight zero blocks of $\mrL_{V^{\otimes n}}$, $\mrR_{V^{\otimes n},V}$ and $\mrR_{V,V^{\otimes n}}$, respectively, and set 
\begin{gather*}
\mbbL_{V^{\otimes n}}:=\Lz_{V^{\otimes n}}^{-1} \mrL_{V^{\otimes n}},\\
\Rp_{V^{\otimes n},V}:=\Rd_{V^{\otimes n},V}^{-1}\mrR_{V^{\otimes n},V} \quad \text{ and }\quad \Rp_{V, V^{\otimes n}}:=\Rd_{V, V^{\otimes n}}^{-1}\mrR_{V, V^{\otimes n}}.
\end{gather*}
\begin{lemma}\label{L:LVn}
For each $n>0$, the matrix $\mbbL_{V^{\otimes n}}$ satisfies the relations 
\begin{align*}
\Rp_{V, V^{\otimes n}} \left(\Rd_{V, V^{\otimes n}} \mbbL_V\Rd_{V, V^{\otimes n}}^{-1}  \right) \mbbL_{V^{\otimes n}}
&=
\mbbL_{V^{\otimes n}}\left(\Rd_{V, V^{\otimes n}}^{-1} \mbbL_V\Rd_{V, V^{\otimes n}}\right)\Rp_{V, V^{\otimes n}}, \\
\Rp_{V^{\otimes n},V} \left(\Rd_{V^{\otimes n},V} \mbbL_{V^{\otimes n}}\Rd_{V^{\otimes n},V}^{-1}  \right) \mbbL_V
&=
\mbbL_V\left(\Rd_{V^{\otimes n},V}^{-1} \mbbL_{V^{\otimes n}}\Rd_{V^{\otimes n},V} \right)\Rp_{V^{\otimes n},V}.
\end{align*}
\end{lemma}
\begin{proof}
These are both generalizations of the first identity of Lemma \ref{L:mbbL}.  The first relation in \eqref{RLL:n} is equivalent to 
\begin{equation*}
\Rd_{V^{\otimes n},V}\Rp_{V^{\otimes n},V} \Lz_{V^{\otimes n}}\mbbL_{V^{\otimes n}}\Lz_V\mbbL_V= \Lz_V\mbbL_V \Lz_{V^{\otimes n}}\mbbL_{V^{\otimes n}}\Rd_{V^{\otimes n},V}\Rp_{V^{\otimes n},V}.
\end{equation*}
The right-hand side is 
\begin{align*}
\Lz_V\Lz_{V^{\otimes n}} &(\Lz_{V^{\otimes n}}^{-1}\mbbL_V \Lz_{V^{\otimes n}}) \mbbL_{V^{\otimes n}}\Rd_{V^{\otimes n},V}\Rp_{V^{\otimes n},V}\\
&=
\Lz_V\Lz_{V^{\otimes n}} (\Rd_{V^{\otimes n},V}\mbbL_V \Rd_{V^{\otimes n},V}^{-1}) \mbbL_{V^{\otimes n}}\Rd_{V^{\otimes n},V}\Rp_{V^{\otimes n},V}\\
&
=\Lz_V\Lz_{V^{\otimes n}} \Rd_{V^{\otimes n},V} \cdot \mbbL_V\left(\Rd_{V^{\otimes n},V}^{-1} \mbbL_{V^{\otimes n}}\Rd_{V^{\otimes n},V} \right)\Rp_{V^{\otimes n},V},
\end{align*}
where we have used that $\Rd_{V^{\otimes n},V}\mbbL_V \Rd_{V^{\otimes n},V}^{-1}=\Lz_{V^{\otimes n}}^{-1}\mbbL_V \Lz_{V^{\otimes n}}$. Similarly, the left-hand side is 
\begin{align*}
\Rd_{V^{\otimes n},V}&\Rp_{V^{\otimes n},V} \Lz_{V^{\otimes n}}\Lz_V\left(\Lz_V^{-1}\mbbL_{V^{\otimes n}}\Lz_V\right)\mbbL_V\\
&
=\Rd_{V^{\otimes n},V}\Rp_{V^{\otimes n},V} \Lz_{V^{\otimes n}}\Lz_V\left(\Rd_{V^{\otimes n},V}\mbbL_{V^{\otimes n}}\Rd_{V^{\otimes n},V}^{-1}\right)\mbbL_V
\\
&
=\Rd_{V^{\otimes n},V}\Lz_V \Lz_{V^{\otimes n}}\cdot \Rp_{V^{\otimes n},V}\left(\Rd_{V^{\otimes n},V}\mbbL_{V^{\otimes n}}\Rd_{V^{\otimes n},V}^{-1}\right)\mbbL_V.
\end{align*}
Comparing these two computations yields the first relation of the lemma. The second relation is proven similarly. \qedhere
\end{proof}

 Let us now define $\mathscr{R}^{\beta}_{\mathbb{V}}\in \End(V\otimes \mathbb{V})$, $_{\mathbb{V}}\mathscr{R}_{\beta}\in  \End(\mathbb{V}\otimes V)$ and $\mathbb{L}_\beta^\mbbV\in \Hom(\mbbV, \mbbV\otimes \URg$, for any $\beta\in \msQ_+$, by 
\begin{gather*}
 \mathbb{L}_\beta^\mbbV:=\sum_{n\in \N} (\mathbb{L}_{V^{\otimes n}})_\beta \circ \mathbf{1}_{V^{\otimes n}},\\
 \mathscr{R}^{\beta}_{\mathbb{V}}:=\sum_{n\in \N} (\Rp_{V,V^{\otimes n}})_\beta \circ \mathbf{1}_{V,V^{\otimes n}}
 \quad \text{ and }\quad 
  _{\mathbb{V}}\mathscr{R}_{\beta}:=\sum_{n\in \N} (\Rp_{V^{\otimes n},V})_\beta \circ \mathbf{1}_{V^{\otimes n},V},
\end{gather*}
where $\mathbf{1}_{V^{\otimes n}}:\mbbV\to V^{\otimes n}$ is the natural projection, $\mathbf{1}_{V,V^{\otimes n}}=\mathrm{Id}_V\otimes \mathbf{1}_{V^{\otimes n}}$ and $\mathbf{1}_{V^{\otimes n},V}=\mathbf{1}_{V^{\otimes n}}\otimes \mathrm{Id}_V$. The elements $\mathscr{R}^{\beta}_{\mathbb{V}}$ and $_{\mathbb{V}}\mathscr{R}_{\beta}$ are defined precisely so that the statement of the following lemma holds. 
\begin{lemma}\label{L:R-mbbV}
For each $\beta\in \msQ_+$, the elements $\mathscr{R}^{\beta}_{\mathbb{V}}$ and $  _{\mathbb{V}}\mathscr{R}_{\beta}$ satisfy
\begin{equation*}
\mathscr{R}^{\beta}_{\mathbb{V}}=(\pi\otimes \pi_{\mbbV})(\mathscr{R}^+_{\beta}) \quad \text{ and }\quad 
  _{\mathbb{V}}\mathscr{R}_{\beta}=(\pi_\mbbV\otimes \pi)(\mathscr{R}^+_{\beta}),
\end{equation*}
where $\mathscr{R}_\beta^+\in U_q(\mfn^+)_\beta\otimes U_q(\mfn^-)_{-\beta}$ is as in Section \ref{ssec:Uqg-R}.
\end{lemma}
\begin{proof} This follows from repeated application of the coproduct identities \eqref{R^+:Hopf} and the relation \eqref{D-Psi}. \qedhere
\end{proof}
Next, we establish a variant of Lemma \ref{L:Rec} and Theorem \ref{T:Gen} which holds for the matrices $\mbbL_\beta^\mbbV$. Recall that $\mbU_\mrR^\lambda(\mfn^+)$ denotes the subalgebra of $\URg$ generated by $\{\Xr_i\}_{i\in \mbI}$. 
\begin{proposition}\label{P:Big-Rec}
For each $\beta\in \msQ_+$, we have 
 \begin{equation*}
 \mbbL_\beta^\mbbV\in \pi_\mbbV(U_q(\mfn^+)_\beta)\otimes \mbU_\mrR^\lambda(\mfn^+)_\beta
 \end{equation*}
 with $\mbbL_{\alpha_i}^\mbbV=\pi_\mbbV(E_i)\otimes \Xr_i$ for each $i\in \mbI$. 
Moreover, these elements satisfy 
\begin{equation}\label{Rec:1'}
\begin{aligned}
 &[\mathscr{R}^{\alpha_i}_{\mathbb{V}},\mathbb{L}_\beta^\mbbV]\\
&=\mathbb{L}_{\beta-\alpha_i}^\mbbV (\mathbb{L}_{\alpha_i})_1\pi_\mbbV(K_i^{-1}) - q^{(\alpha_i,\beta-\alpha_i)}(\mathbb{L}_{\alpha_i})_1  \mathbb{L}_{\beta-\alpha_i}^\mbbV\pi_\mbbV(K_i)
\end{aligned}
\end{equation}
in $\pi(U_q(\mfn^+))\otimes \pi_\mbbV(U_q(\mfg))\otimes \URg$.
\end{proposition}
\begin{proof}

The proof that \eqref{Rec:1'} is satisfied follows the same argument as used to establish the relation \eqref{Rec:1} in Lemma \ref{L:Rec}.  
Taking the $\End(V)_{\alpha_i}\otimes \End(V^{\otimes n})_{\beta-\alpha_i}$ component of the first relation of Lemma \ref{L:LVn} yields
\begin{align*}
 (\Rp_{V, V^{\otimes n}})_{\alpha_i}&  (\mbbL_{V^{\otimes n}})_{\beta}+
 \left(\Rd_{V, V^{\otimes n}} (\mbbL_V)_{\alpha_i}\Rd_{V, V^{\otimes n}}^{-1}  \right) (\mbbL_{V^{\otimes n}})_{\beta-\alpha_i}\\
& 
=
(\mbbL_{V^{\otimes n}})_{\beta_i-\alpha_i}\left(\Rd_{V, V^{\otimes n}}^{-1} (\mbbL_V)_{\alpha_i}\Rd_{V, V^{\otimes n}}\right)
+
(\mbbL_{V^{\otimes n}})_\beta(\Rp_{V, V^{\otimes n}})_{\alpha_i}.
\end{align*}
 Since $\Rd_{V, V^{\otimes n}}^{-1} (\mbbL_V)_{\alpha_i}\Rd_{V, V^{\otimes n}}$ coincides with $(\mbbL_V)_{\alpha_i} \pi_{V^{\otimes n}}(K_i^{-1})$, this is equivalent to 
\begin{equation*}
 \begin{aligned}
 &[(\Rp_{V, V^{\otimes n}})_{\alpha_i}, (\mbbL_{V^{\otimes n}})_\beta]\\
 &=
 (\mbbL_{V^{\otimes n}})_{\beta_i-\alpha_i}(\mbbL_V)_{\alpha_i}\pi_{V^{\otimes n}}(K_i^{-1})
 - q^{(\alpha_i,\beta-\alpha_i)}
  (\mbbL_V)_{\alpha_i}(\mbbL_{V^{\otimes n}})_{\beta-\alpha_i}\pi_{V^{\otimes n}}(K_i).
  \end{aligned}
\end{equation*}
 As this holds for all $n$, we can conclude that \eqref{Rec:1'} is satisfied.

 Next, let us establish that $\mbbL_{\alpha_i}^\mbbV=\pi_\mbbV(E_i)\otimes \Xr_i$ for each $i\in \mbI$. By definition of $\mathbb{V}$, it is enough to show that 
\begin{equation*} 
 (\mbbL_{V^{\otimes n}})_{\alpha_i}=\pi_{V^{\otimes n}}(E_i)\otimes \Xr_i
 \quad \forall \quad n\in \N \; \text{ and }\;  i\in \mbI.
\end{equation*}
For each $1\leq j\leq n$, set $\Lz_{j+1,...,n}:=\Lz_{j+1}\cdots \Lz_n\in \End(V)^{\otimes n}\otimes \URg$, with the understanding that $\Lz_{j+1,...,n}=1$ if $j=n$. Then the definition of $\mbbL_{V^{\otimes n}}$ yields that 
\begin{align*}
(\mbbL_{V^{\otimes n}})_{\alpha_i}&=\sum_{j=1}^n \Lz_{j+1,...,n}^{-1} (\mbbL_j)_{\alpha_i} \Lz_{j+1,...,n}\\
&=
\sum_{j=1}^n (\mbbL_j)_{\alpha_i} \pi(K_i)_{j+1}\cdots \pi(K_i)_{n}\\
&
=\sum_{j=1}^n \pi(E_i)_j \pi(K_i)_{j+1}\cdots \pi(K_i)_{n} \otimes \Xr_i
=\pi_{V^{\otimes n}}(E_i)\otimes \Xr_i.
\end{align*}
The proof that, more generally, one has $\mbbL_\beta^\mbbV\in \pi_\mbbV(U_q(\mfn^+)_\beta)\otimes \mbU_\mrR^\lambda(\mfn^+)_\beta$ for any $\beta\in \msQ_+$ now proceeds following a simple generalization of the argument given in the proof of Theorem \ref{T:Gen}. To begin, taking the $\End(V^{\otimes n})_\beta\otimes \End(V)_0$ component of the second relation in Lemma \ref{L:LVn} yields
 \begin{align*}
  \Rd_{V^{\otimes n},V} &(\mbbL_{V^{\otimes n}})_{\beta}\Rd_{V^{\otimes n},V}^{-1}+(\Rp_{V^{\otimes n},V})_\beta  (\mbbL_V)_\beta\\
  &+
 \sum_{\alpha\in \dot\msQ_+^\beta} (\Rp_{V^{\otimes n},V})_\alpha \left(\Rd_{V^{\otimes n},V} (\mbbL_{V^{\otimes n}})_{\beta-\alpha}\Rd_{V^{\otimes n},V}^{-1}  \right) (\mbbL_V)_\alpha\\
&=\Rd_{V^{\otimes n},V}^{-1} (\mbbL_{V^{\otimes n}})_{\beta}\Rd_{V^{\otimes n},V} +
(\mbbL_V)_\beta(\Rp_{V^{\otimes n},V})_\beta\\
&\quad+\sum_{\alpha\in \dot\msQ_+^\beta}(\mbbL_V)_\alpha\left(\Rd_{V^{\otimes n},V}^{-1} (\mbbL_{V^{\otimes n}})_{\beta-\alpha}\Rd_{V^{\otimes n},V} \right)(\Rp_{V^{\otimes n},V})_\alpha,
 \end{align*}
where we recall that $\dot\msQ_+^\beta$ is defined at the end of Section \ref{ssec:g}.
%
 This implies that in $\mathrm{Hom}(\mbbV\otimes V, \mathbb{V}\otimes V\otimes \URg)$ one has
 \begin{equation}\label{Rec:2'}
 \begin{gathered}
 \mbbL_{\beta}^\mbbV\,\pi_V(K_\beta-K_\beta^{-1})=[(\mbbL_V)_\beta,\, _{\mbbV}\mathscr{R}_\beta]-\sum_{\alpha\in \dot\msQ_+^\beta}\mbbL_{\beta,\alpha}^{\mbbV},\\
 \mbbL_{\beta,\alpha}^{\mbbV}:=\, _{\mbbV}\mathscr{R}_\alpha \cdot \mbbL^\mbbV_{\beta-\alpha}\pi_V(K_{\beta-\alpha}) (\mbbL_V)_\alpha - (\mbbL_V)_\alpha \mbbL^{\mbbV}_{\beta-\alpha}\pi_V(K_{\beta-\alpha}^{-1})\cdot\, _{\mbbV}\mathscr{R}_\alpha.
 \end{gathered}
 \end{equation}
By Lemma \ref{L:R-mbbV} and Theorem \ref{T:Gen}, we have $_{\mbbV}\mathscr{R}_\alpha\in \pi_\mbbV(U_q(\mfn^+)_\alpha)\otimes \pi(U_q(\mfn^-)_{-\alpha})$ and  $(\mbbL_V)_\alpha \in \pi(U_q(\mfn^+)_\alpha)\otimes \mbU_\mrR^\lambda(\mfn^+)_\alpha$ for each $\alpha\in \msQ_+$. Using these facts and induction on the height of $\beta$, one deduces from \eqref{Rec:2'} that $\mbbL_\beta^\mbbV\in \pi_\mbbV(U_q(\mfn^+)_\beta)\otimes \mbU_\mrR^\lambda(\mfn^+)_\beta$ for all $\beta\in \msQ_+$, as desired. 
\qedhere
 \end{proof}
For each $i\in \mbI$,  introduce a degree $\alpha_i$ linear endomorphism $\mathrm{ad}_{q,i}$ of $\URg$ by 
 \begin{equation*}
 \mathrm{ad}_{q,i}(y)=y\Xr_i-q^{(\alpha_i,\beta)}\Xr_iy \in \URg_{\beta+\alpha_i}\quad \forall \quad y\in \URg_\beta.
 \end{equation*}
Since $\pi_\mbbV$ is injective, the triangular decomposition \eqref{Uqg:TD} for $U_q(\mfg)$ implies that there is a unique $\Q(q)$-linear map $\varepsilon_\mbbV^+:\pi_\mbbV(U_q(\mfg))\to \pi_\mbbV(U_q(\mfn^+))$ satisfying 
 \begin{equation}\label{veps+}
 \varepsilon_\mbbV^+(\pi_\mbbV(xy))=\pi_\mbbV(x)\veps(y) \quad \forall \quad x\in U_q(\mfn^+),\, y\in U_q(\mfb^-),
 \end{equation}
 where $\veps:U_q(\mfg)\to \Q(q)$ is the counit. We then have the following result.  
\begin{proposition}\label{P:Almost-Serre}
For each $i\in \mbI$, $k\geq 1$ and $\beta\in \msQ_+$, we have 
\begin{equation*}
\veps_\mbbV^+ \cdot \mathrm{ad}(\pi_\mbbV(F_i))^k\left(\mbbL_\beta^\mbbV\right)
=
\frac{1}{(q_i-q_i^{-1})^k}\mathrm{ad}_{q,i}^k \left(\mbbL_{\beta-k\alpha_i}^\mbbV\right)
\end{equation*}
in $\pi_\mbbV(U_q(\mfn^+))\otimes \URg$. Moreover, if $\beta=k\alpha_i+\alpha_j$ for some $j\neq i$ and $k=1-a_{ij}$, then the left-hand side of the above equality vanishes and therefore
\begin{equation*}
\mathrm{ad}_{q,i}^{1-a_{ij}}(\Xr_j)=0. 
\end{equation*}

\end{proposition}

\begin{proof}
For elements $x\in \URg_\alpha$ and $y\in\URg_\beta$, let us introduce the $q$-bracket 
 \begin{equation*}
 [x,y]_q:=xy-q^{(\alpha,\beta)}yx \in \URg_{\beta+\alpha}. 
 \end{equation*}
 By Lemma \ref{L:R-mbbV} we have $\mathscr{R}_\mbbV^{\alpha_i}=(q_i-q_i^{-1}) \pi(E_i)\otimes \pi_\mbbV(F_i)$ and hence the main identity of the proposition is equivalent to the following relation in $\pi(U_q(\mfn^+))^{\otimes k}\otimes \pi_\mbbV(U_q(\mfg))\otimes \URg$:
\begin{equation}\label{magic-serre}
\begin{aligned}
\veps_\mbbV^+ &\left[ (\mathscr{R}_\mbbV^{\alpha_i})_{1,k+1},\left[(\mathscr{R}_\mbbV^{\alpha_i})_{2,k+1},\cdots \left[ (\mathscr{R}_\mbbV^{\alpha_i})_{k,k+1}, \mbbL_\beta^\mbbV\right]\cdots \right]\right]\\
&
=
\left[ \cdots \left[ \left[ \mbbL^\mbbV_{\beta-k\alpha_i},(\mbbL_{\alpha_i})_1\right]_q, (\mbbL_{\alpha_i})_2\right]_q, \ldots,(\mbbL_{\alpha_i})_k\right]_q.
\end{aligned}
\end{equation}
We prove this identity by induction on $k$ using the relation \eqref{Rec:1'} established in the previous proposition. When $k=1$, that relation implies immediately that 
\begin{equation*}
  \begin{aligned}
\veps_\mbbV^+ [&\mathscr{R}_{\mbbV}^{\alpha_i}, \mbbL_\beta^\mbbV]\\
 &=
 \mbbL_{\beta_i-\alpha_i}^\mbbV(\mbbL_{\alpha_i})_1
 - q^{(\alpha_i,\beta-\alpha_i)}
  (\mbbL_{\alpha_i})_1\mbbL^{\mbbV}_{\beta-\alpha_i}=[\mbbL^{\mbbV}_{\beta-\alpha_i},(\mbbL_{\alpha_i})_1]_q,
  \end{aligned}
  \end{equation*}
as desired.   Suppose now that \eqref{magic-serre} holds for some fixed $k$. As $\mathrm{Ker}(\veps_\mbbV^+)$ is preserved by $\mathrm{ad}(\pi_\mbbV(F_i))$, to establish \eqref{magic-serre} for $k+1$ it suffices to show that 
\begin{equation*}
\begin{aligned}
\veps_\mbbV^+ &\left[ (\mathscr{R}_\mbbV^{\alpha_i})_{1,k+2},\left[ \cdots \left[ \left[ \mbbL^\mbbV_{\beta-k\alpha_i},(\mbbL_{\alpha_i})_2\right]_q, (\mbbL_{\alpha_i})_3\right]_q, \ldots,(\mbbL_{\alpha_i})_{k+1}\right]_q\right]\\
&
=
\left[ \cdots \left[ \left[ \mbbL^\mbbV_{\beta-(k+1)\alpha_i},(\mbbL_{\alpha_i})_1\right]_q, (\mbbL_{\alpha_i})_2\right]_q, \ldots,(\mbbL_{\alpha_i})_k\right]_q.
\end{aligned}
\end{equation*}
Since the left-hand side is 
\begin{equation*}
\veps_\mbbV^+ \left[ \cdots \left[ \left[ [(\mathscr{R}_\mbbV^{\alpha_i})_{1,k+2},\mbbL^\mbbV_{\beta-k\alpha_i}],(\mbbL_{\alpha_i})_2\right]_q, (\mbbL_{\alpha_i})_3\right]_q, \ldots,(\mbbL_{\alpha_i})_{k+1}\right]_q,
\end{equation*}
this follows by again applying \eqref{Rec:1'}. Thus, we may conclude that \eqref{magic-serre}, and therefore the main identity of the proposition, holds for all $k\geq 1$.

Suppose now that $j\in \mbI$ is such that $j\neq i$, and set $\beta=(1-a_{ij})\alpha_i+\alpha_j$. To complete the proof of the proposition, we must explain why 
\begin{equation*}
\veps_\mbbV^+ \cdot \mathrm{ad}(\pi_\mbbV(F_i))^{1-a_{ij}}\left(\mbbL_\beta^\mbbV\right)=0.
\end{equation*}
Since $\pi_{\mbbV}(U_q(\mfn^+)_{\beta})$ is spanned by the elements $\pi_\mbbV(E_i^r E_j E_i^{1-a_{ij}-r})$ for $0\leq r<1-a_{ij}$, the above equality is a consequence of the following more general claim:

\noindent \textit{Claim}. Given $i\neq j$ in $\mbI$ and $k\geq 1$, set $E_{i,j}^{(r)}:=E_i^r E_j E_i^{k-r}\in U_q(\mfg)$. Then
 \begin{equation*}
 \mathrm{ad}(F_i)^k(E_{i,j}^{(r)})\in \mathrm{Ker}(\veps^+)\quad \forall \quad 0\leq r< k,
 \end{equation*}
 where $\veps^+:U_q(\mfg)\to U_q(\mfn^+)$ is the unique linear map satisfying $\pi_\mbbV\circ \veps^+=\veps_\mbbV^+\circ \pi_\mbbV$. 
 
\begin{proof}[Proof of claim]
First note that for each $\ell\in \mbI$, we have  
 \begin{equation*}
 F_\ell E_i=\delta_{\ell,i}\frac{K_i^{-1}-K_i}{q_i-q_i^{-1}}+E_i F_\ell \in \mathrm{Ker}(\veps^+).
 \end{equation*}
 Since $\mathrm{Ker}(\veps^+)$ is invariant under $\mathrm{ad}(F_i)$, we obtain 
 \begin{equation*}
 \mathrm{ad}(F_i)^{k}(E_{i,j}^{(r)})= \mathrm{ad}(F_i)^{k}(XE_i)=\mathrm{ad}(F_i)^{k}(X)E_i \mod \mathrm{Ker}(\veps^+)
 \end{equation*}
 for all $0\leq r<k$, where $X=E_i^r E_j E_i^{k-r-1}$.  Note that $\mathrm{ad}(F_i)^{k}(X)$ has degree $\alpha_j-\alpha_i$, and so belongs to the subspace $U_q(\mfb^+)\otimes \mcJ_i$ of $U_q(\mfg)$, where $\mcJ_i$ is the two sided ideal of $U_q(\mfn^-)$ generated by $F_i$. Since $F_\ell E_i\in \mathrm{Ker}(\veps^+)$ for all $\ell\in \mbI$, we can conclude that $\mathrm{ad}(F_i)^{k}(X)E_i\in \mathrm{Ker}(\veps^+)$. This implies the statement of the claim, and thus completes the proof of the proposition. \qedhere
 \end{proof} \let\qed\relax
\end{proof}
The $q$-Serre relation \eqref{main-suff:5} can now be deduced as a simple corollary of the second assertion of Proposition \ref{P:Almost-Serre}.
\begin{corollary}\label{C:Serre}
Let $i,j\in \mbI$ with $i\neq j$. Then $\Xr_i$ and $\Xr_j$ satisfy the relation \eqref{main-suff:5}:
\begin{equation*}
\sum_{r=0}^{1-a_{ij}} (-1)^r \sbinom{1-a_{ij}}{r}_{q_i}\Xr_i^r \Xr_j \Xr_i^{1-a_{ij}-r}=0. 
\end{equation*}
\end{corollary}
\begin{proof}
This follows from the relation $\mathrm{ad}_{q,i}^{1-a_{ij}}(\Xr_j)=0$ of Proposition \ref{P:Almost-Serre} and the following identity, which is a simple application of the $q$-analogue of Pascal's identity (see \cite{Jantzen-Book}*{\S0.2}): For each $i,j\in \mbI$ with $i\neq j$ and $k\geq 0$, we have 
\begin{equation*}
\mathrm{ad}_{q,i}^k(\Xr_j)=\sum_{r=0}^k (-1)^r q_i^{-r(1-a_{ij}-k)}\sbinom{k}{r}_{q_i}\Xr_i^r \Xr_j \Xr_i^{k-r}. \qedhere
\end{equation*}
\end{proof}

\subsection{Recovering \texorpdfstring{$U_q(\mfg)$}{U\_q(g)}}\label{ssec:Cons-R}
 By Theorem \ref{T:main}, the subalgebra of $\URg$ generated by $\{\upxi_i^{\pm 1},\Xr_i,\Yr_i\}_{i\in \mbI}$ is isomorphic to the Drinfeld--Jimbo algebra $U_q(\mfg)$. In this section, we provide equivalent characterizations of this subalgebra which are natural from the point of view of the $R$-matrix formalism. 
 To begin, following \eqref{y:1}, we define 
\begin{equation*}
\EuScript{Z}:=(\mathscr{l}_\lambda)^{\olambda} \upxi_{\olambda \lambda}=(\mathscr{l}_\lambda)^{\olambda}\prod_{j\in \mbI}\upxi_j^{n_j}\in \URg,
\end{equation*}
where $n_j=d_j^{-1}(\varpi_j, \olambda \lambda)$ for each $j\in \mbI$. By Theorem \ref{T:Lz} and \eqref{Ad-upxi}, this element belongs to the center of $\URg$. In fact, it follows readily from this definition that
\begin{equation*}
\Upupsilon(\EuScript{Z})=v_\lambda=v^{\olambda}.
\end{equation*}
Next, mimicking the notation from Section \ref{ssec:bUqg-alg}, we introduce the subalgebra 
\begin{equation*}
\URg^\upchi:=\{x\in \URg: \upchi_\zeta(x)=x\quad \forall \; \zeta\in \Q(q)^\times\}\subset \URg,
\end{equation*}
where $\upchi_\zeta$ is as in Proposition \ref{P:R-aut}. From the definition of the isomorphism $\Upupsilon$ given in Theorem \ref{T:main}, we see  that 
\begin{equation*}
\upchi_{q,\zeta}\circ \Upupsilon=\Upupsilon\circ \upchi_\zeta \quad \forall\quad \zeta\in \Q(q)^\times,
\end{equation*}
where $\upchi_{q,\zeta}$ is the automorphism of $\bUqg$ defined in Section \ref{ssec:bUqg-alg}. In particular, $\Upupsilon$ restricts to an isomorphism between $\URg^\upchi$ and the subalgebra of $\bUqg$ consisting of elements fixed by each $\upchi_{q,\zeta}$.  This discussion, combined with Parts \eqref{Uqlambda:3} and \eqref{Uqlambda:4} of Proposition \ref{P:bUqg-str}, admits the following corollary.
\begin{corollary}\label{C:Recover}
$\Upupsilon$ gives rise to Hopf algebra isomorphisms 
\begin{equation*}
\Upupsilon|_{\URg^\upchi}:\URg^\upchi\iso U_q(\mfg) \quad \text{ and }\quad \bar{\Upupsilon}:\URg/(\EuScript{Z}-1)\iso  U_q^\lambda(\mfg).
\end{equation*}
\end{corollary}
{
In particular, $\URg^\upchi$ is generated by the set of elements $\{\upxi_i^{\pm 1},\Xr_i,\Yr_i\}_{i\in \mbI}$. Another system of generators can be obtained by normalizing $\mrL^+$ and $\mrL^-$ so that they act as the identity operators on the highest and lowest  weight spaces of $V$, respectively. In more detail, let $w$ denotes the longest element of the Weyl group of $\mfg$, and set
\begin{equation*}
\mathit{L}^+:=\mathscr{l}_\lambda^{-1} \mrL^+ \quad \text{ and }\quad \mathit{L}^-:=\mrL^- \mathscr{l}_{w(\lambda)}.
\end{equation*} 
Then $\mathit{L}^+|_{V_\lambda}=\mathrm{Id}_{V_\lambda}$, $\mathit{L}^-|_{V_{w(\lambda)}}=\mathrm{Id}_{V_{w(\lambda)}}$, and it is not difficult to see that the coefficients of $\mathit{L}^\pm$ generate $\URg^\upchi\cong U_q(\mfg)$ as an algebra.

The larger subalgebra  $U_q(\mfg)\otimes \Q(q)[v_\lambda^{\pm 1}]$ of $\bUqg$ considered in Part \eqref{Uqlambda:2} of Proposition \ref{P:bUqg-str} also admits a natural characterization within $\URg$, as we now explain. Set $n:=\olambda$, and define
\begin{equation*}
\mrL^\pm_{V^{\otimes n}}:=\mrL_1^\pm \mrL_2^\pm \cdots \mrL_n^\pm \in \End(V)^{\otimes n}\otimes \URg
\end{equation*}
so that $\mrL^+_{V^{\otimes n}}$ coincides with $\mrL_{V^{\otimes n}}$ defined above \eqref{RLL:n}. 
Similarly, we define $\scK_{V^{\otimes n}}$ in $\End(V)^{\otimes n}\otimes U_q(\mfg)$ by  
\begin{equation*}
\scK_{V^{\otimes n}}:=\scK_1\scK_2\cdots \scK_n= \sum_{\mu} \mathrm{Id}_{(V^{\otimes n})_\mu} \otimes K_{\mu},
\end{equation*}
where the sum is taken over the set of weights of $V^{\otimes n}$, which is a subset of the root lattice $\msQ$ since $n=\olambda$. We then have the following corollary, which is deduced from the definitions of $\EuScript{Z}$ and $\Upupsilon$ with the help of the relations \eqref{R^+:Hopf} and \eqref{Uppsi-K}; see also Lemma \ref{L:R-mbbV} and the computation of $(\mbbL_{V^{\otimes n}})_{\alpha_i}$ in the proof of  Proposition \ref{P:Big-Rec}. 
\begin{corollary}\label{C:nlambda}
Let $\dot\mbU_\mrR^{n\lambda}(\mfg)$ denote the subalgebra of $\URg$ generated by the coefficients of $\mrL_{V^{\otimes n}}^+$ and $\mrL_{V^{\otimes n}}^-$. Then $\Upupsilon$ restricts to an isomorphism
\begin{equation*}
\Upupsilon|_{\dot\mbU_\mrR^{n\lambda}(\mfg)}:\dot\mbU_\mrR^{n\lambda}(\mfg)\iso U_q(\mfg)\otimes \Q(q)[v_\lambda^{\pm 1}]. 
\end{equation*}
Moreover, the coefficients of the normalized matrices $\msL_{V^{\otimes n}}^\pm:=\EuScript{Z}^{\mp 1}\mrL_{V^{\otimes n}}^\pm$ generate $\URg^\upchi\cong U_q(\mfg)$ and one has 
\begin{equation*}
\Upupsilon(\msL_{V^{\otimes n}}^+)= \scK_{V^{\otimes n}}^{-1}\cdot (\pi_{V^{\otimes n}}\otimes \omega)(\mathscr{R}^+), \quad
\Upupsilon(\msL_{V^{\otimes n}}^-)=(\pi_{V^{\otimes n}}\otimes \omega)((\mathscr{R}^+_{21})^{-1})\cdot \scK_{V^{\otimes n}}.
\end{equation*}
\end{corollary}

\begin{remark}
Since $n\lambda$ is a nonzero dominant integral weight, we may consider the quantum group 
$\mathbf{U}_\mrR^{n\lambda}(\mfg)$ obtained from Definition \ref{D:bUqg} by replacing $V$ by the irreducible summand $V(n\lambda)$ of $V^{\otimes n}$. Since, by  Part \eqref{Uqlambda:2} of Proposition \ref{P:bUqg-str}, one has $\mbU_q^{n\lambda}(\mfg)\cong U_q(\mfg)\otimes \Q(q)[v_\lambda^{\pm 1}]$, Theorem \ref{T:main} outputs an isomorphism of $\msQ$-graded Hopf algebras
\begin{equation*}
 \Upupsilon_{n\lambda}: \mathbf{U}_\mrR^{n\lambda}(\mfg) \iso U_q(\mfg)\otimes \Q(q)[v_\lambda^{\pm 1}].
\end{equation*}
The composite  $\Upupsilon|_{\dot\mbU_\mrR^{n\lambda}(\mfg)}^{-1}\circ \Upupsilon_{n\lambda}$ then provides an isomorphism $\mathbf{U}_\mrR^{n\lambda}(\mfg)\iso \dot\mbU_\mrR^{n\lambda}(\mfg)$ which satisfies 
\begin{equation*}
(\Upupsilon|_{\dot\mbU_\mrR^{n\lambda}(\mfg)}^{-1}\circ \Upupsilon_{n\lambda})(\mrL_{V(n\lambda)}^\pm)
=\mathbf{1}_{V(n\lambda)} \circ \mrL^\pm_{V^{\otimes n}} \circ \imath_{V(n\lambda)},
\end{equation*}
where $\mrL^\pm_{V(n\lambda)}$ are the generating matrices $\mrL^\pm$ for $\mbU_\mrR^{n\lambda}(\mfg)$,  $\mathbf{1}_{V(n\lambda)}\in \End_{U_q(\mfg)}(V^{\otimes n})$ is the projection onto the summand $V(n\lambda)$ of $V^{\otimes n}$, and $\imath_{V(n\lambda)}:V(n\lambda)\into V^{\otimes n}$ is the inclusion map. 
\end{remark}





\begin{bibdiv}
\begin{biblist}

\bib{CPBook}{book}{
      author={Chari, V.},
      author={Pressley, A.},
       title={A guide to quantum groups},
   publisher={Cambridge University Press, Cambridge},
        date={1994},
        ISBN={0-521-43305-3},
}

\bib{DCP-survey}{incollection}{
      author={De~Concini, C.},
      author={Procesi, C.},
       title={Quantum groups},
        date={1993},
   booktitle={{$D$}-modules, representation theory, and quantum groups
  ({V}enice, 1992)},
      series={Lecture Notes in Math.},
      volume={1565},
   publisher={Springer, Berlin},
       pages={31\ndash 140},
}

\bib{Dr}{article}{
      author={Drinfel'd, V.},
       title={Hopf algebras and the quantum {Y}ang-{B}axter equation},
        date={1985},
     journal={Soviet Math. Dokl.},
      volume={32},
      number={1},
       pages={254\ndash 258},
}

\bib{Dr-almost}{article}{
      author={Drinfel'd, V.},
       title={Almost cocommutative {H}opf algebras},
        date={1989},
        ISSN={0234-0852},
     journal={Algebra i Analiz},
      volume={1},
      number={2},
       pages={30\ndash 46},
        note={English transl., Leningrad Math. J. \textbf{1} (1990), no. 2,
  321–342},
}

\bib{DF93}{article}{
      author={Ding, J.},
      author={Frenkel, I.},
       title={Isomorphism of two realizations of quantum affine algebra
  {$U_q({\germ g}{\germ l}(n))$}},
        date={1993},
        ISSN={0010-3616},
     journal={Comm. Math. Phys.},
      volume={156},
      number={2},
       pages={277\ndash 300},
}

\bib{FRT}{article}{
      author={Faddeev, L.},
      author={Reshetikhin, N.},
      author={Takhtajan, L.},
       title={Quantization of {L}ie groups and {L}ie algebras},
        date={1990},
     journal={Leningrad Math. J.},
      volume={1},
      number={1},
       pages={193\ndash 225},
}

\bib{FiTs19}{article}{
      author={Finkelberg, M.},
      author={Tsymbaliuk, A.},
       title={Shifted quantum affine algebras: integral forms in type {$A$}},
        date={2019},
        ISSN={2199-6792},
     journal={Arnold Math. J.},
      volume={5},
      number={2-3},
       pages={197\ndash 283},
}

\bib{GRW-RQUE}{unpublished}{
      author={Gautam, S.},
      author={Rupert, M.},
      author={Wendlandt, C.},
       title={The {R}-matrix formalism for quantized enveloping algebras},
        date={2022},
        note={To appear in Annales de l'Institut Fourier. {\tt
  arXiv:2210.06770}},
}

\bib{GRWvrep}{article}{
      author={Guay, N.},
      author={Regelskis, V.},
      author={Wendlandt, C.},
       title={Vertex representations for {Y}angians of {K}ac-{M}oody algebras},
        date={2019},
     journal={J. \'{E}c. polytech. Math.},
      volume={6},
       pages={665\ndash 706},
}

\bib{Humphreys-book}{book}{
      author={Humphreys, J.},
       title={Introduction to {L}ie algebras and representation theory},
      series={Graduate Texts in Mathematics, Vol. 9},
   publisher={Springer-Verlag, New York-Berlin},
        date={1972},
}

\bib{MiHa-L98}{article}{
      author={Hayaishi, N.},
      author={Miki, K.},
       title={{$L$} operators and {D}rinfeld's generators},
        date={1998},
     journal={J. Math. Phys.},
      volume={39},
      number={3},
       pages={1623\ndash 1636},
}

\bib{Jantzen-Book}{book}{
      author={Jantzen, J.~C.},
       title={Lectures on quantum groups},
      series={Graduate Studies in Mathematics},
   publisher={American Mathematical Society, Providence, RI},
        date={1996},
      volume={6},
        ISBN={0-8218-0478-2},
}

\bib{Jimbo86}{article}{
      author={Jimbo, M.},
       title={A {$q$}-analogue of {$U({\germ g}{\germ l}(N+1))$}, {H}ecke
  algebra, and the {Y}ang-{B}axter equation},
        date={1986},
        ISSN={0377-9017},
     journal={Lett. Math. Phys.},
      volume={11},
      number={3},
       pages={247\ndash 252},
}

\bib{JLM20-c}{article}{
      author={Jing, N.},
      author={Liu, M.},
      author={Molev, A.},
       title={Isomorphism between the {$R$}-matrix and {D}rinfeld presentations
  of quantum affine algebra: type {$C$}},
        date={2020},
        ISSN={0022-2488},
     journal={J. Math. Phys.},
      volume={61},
      number={3},
       pages={031701, 41},
}

\bib{JLM20-b}{article}{
      author={Jing, N.},
      author={Liu, M.},
      author={Molev, A.},
       title={Isomorphism between the {$R$}-matrix and {D}rinfeld presentations
  of quantum affine algebra: types {$B$} and {$D$}},
        date={2020},
     journal={SIGMA Symmetry Integrability Geom. Methods Appl.},
      volume={16},
       pages={Paper No. 043, 49},
}

\bib{KR90}{article}{
      author={Kirillov, A.},
      author={Reshetikhin, N.},
       title={{$q$}-{W}eyl group and a multiplicative formula for universal
  {$R$}-matrices},
        date={1990},
        ISSN={0010-3616},
     journal={Comm. Math. Phys.},
      volume={134},
      number={2},
       pages={421\ndash 431},
}

\bib{KS-book}{book}{
      author={Klimyk, A.},
      author={Schm\"{u}dgen, K.},
       title={Quantum groups and their representations},
      series={Texts and Monographs in Physics},
   publisher={Springer-Verlag, Berlin},
        date={1997},
        ISBN={3-540-63452-5},
}

\bib{Lusztig-Book}{book}{
      author={Lusztig, G.},
       title={Introduction to quantum groups},
      series={Modern Birkh\"{a}user Classics},
   publisher={Birkh\"{a}user/Springer, New York},
        date={2010},
        ISBN={978-0-8176-4716-2},
         url={https://doi.org/10.1007/978-0-8176-4717-9},
        note={Reprint of the 1994 edition},
}

\bib{LeSo90}{article}{
      author={Levendorski\u{\i}, S.},
      author={Soibelman, Y.},
       title={Some applications of the quantum {W}eyl groups},
        date={1990},
        ISSN={0393-0440},
     journal={J. Geom. Phys.},
      volume={7},
      number={2},
       pages={241\ndash 254},
}

\bib{Pass-14}{article}{
      author={Passman, D.},
       title={Elementary bialgebra properties of group rings and enveloping
  rings: an introduction to {H}opf algebras},
        date={2014},
        ISSN={0092-7872},
     journal={Comm. Algebra},
      volume={42},
      number={5},
       pages={2222\ndash 2253},
}

\bib{Sasaki}{article}{
      author={Sasaki, N.},
       title={Quantization of {L}ie group and algebra of {$G_2$} type in the
  {F}addeev-{R}eshetikhin-{T}akhtajan approach},
        date={1995},
        ISSN={0022-2488},
     journal={J. Math. Phys.},
      volume={36},
      number={8},
       pages={4476\ndash 4488},
}

\bib{WRTT}{article}{
      author={Wendlandt, C.},
       title={The {$R$}-matrix presentation for the {Y}angian of a simple {L}ie
  algebra},
        date={2018},
        ISSN={0010-3616},
     journal={Comm. Math. Phys.},
      volume={363},
      number={1},
       pages={289\ndash 332},
}

\end{biblist}
\end{bibdiv}

\end{document}